\documentclass[12pt, a4paper]{amsart}

\usepackage{amssymb,amsmath,amsthm,mathrsfs}
\usepackage{enumitem}
\usepackage{graphicx,tikz}
\usepackage[colorlinks=true,linkcolor=black,citecolor=black,urlcolor=black]{hyperref}
\usepackage{url}

\numberwithin{equation}{section}
\newcommand{\e}{e}
\newcommand{\ep}{\varepsilon}
\newcommand{\Z}{\mathbb{Z}}
\newcommand{\Q}{\mathbb{Q}}
\newcommand{\R}{\mathbb{R}}
\newcommand{\F}{\mathbb{F}}

\renewcommand{\Re}{\mathrm{Re}}
\renewcommand{\bar}[1]{\overline{#1}}
\renewcommand{\le}{\leqslant}

\renewcommand{\ge}{\geqslant}

\theoremstyle{plain}
\newtheorem{thm}{Theorem}
\newtheorem{lem}{Lemma}[section]
\newtheorem{prop}[lem]{Proposition}
\newtheorem{cor}[lem]{Corollary}

\title[Central non-vanishing of Dirichlet $L$-functions]{Central non-vanishing of Dirichlet $L$-functions}

\author[B. Durkan]{Benjamin Durkan}
\address{Department of Mathematics, The University of Manchester, Oxford Road, Manchester, M13 9PL}
\email{benjamin.durkan@manchester.ac.uk}

\author[A. Pearce-Crump]{Andrew Pearce-Crump}
\address{School of Mathematics, University of Bristol, Fry Building, Woodland Road, Bristol, BS8 1UG, United Kingdom}
\email{andrew.pearce-crump@bristol.ac.uk}

\date{}

\begin{document}

\begin{abstract}
We prove that, for every sufficiently large modulus $q\not\equiv2\bmod4$, at least $3/8-o(1)$ of the primitive Dirichlet characters $\chi$ modulo $q$ have $L(1/2,\chi)\ne0$. We also establish stronger proportions for moduli $q$ satisfying certain arithmetic conditions.
\end{abstract}
\maketitle

\section{Introduction}\label{sec:introduction}

A conjecture attributed to Chowla predicts that $L(1/2,\chi)\ne0$ for every primitive Dirichlet character $\chi$; see \cite{CechMat}.

Let $\varphi^*(q)$ be the number of primitive characters modulo $q$. For $q\not\equiv2\bmod4$, let
$$
\kappa(q)=\frac{1}{\varphi^*(q)}
\#\left\{\chi\bmod q:\chi\hbox{ primitive},\ L(1/2,\chi)\ne0\right\}.
$$
For $j\in\{0,1\}$, let
$$
\mathcal X^*(q;j)=
\{\chi\bmod q:\chi\hbox{ primitive},\ \chi(-1)=(-1)^j\}.
$$
For a nonempty parity class $\mathcal X^*(q;j)$, write
$$
\kappa_j(q)=\frac{1}{\#\mathcal X^*(q;j)}
\#\{\chi\in\mathcal X^*(q;j):L(1/2,\chi)\ne0\}.
$$
\looseness=-1 The excluded congruence class has no primitive characters once $q>2$; we call the remaining moduli admissible. Classical first and second moments give only $\gg\varphi^*(q)/\log q$ non-zero values \cite{Paley}, while Balasubramanian--Murty obtained the first positive proportion \cite{BM}.

Michel and VanderKam's two-piece mollifier \cite{MV}, following Iwaniec--Sarnak \cite{IS}, has a root-number piece whose cross terms limit the lengths. Bui treated general conductors \cite{Bui}. Stronger bilinear Kloosterman estimates extended the lengths for prime conductors \cite{KhanNgo,KMN}; Leung improved the proportion for smooth squarefree conductors \cite{Leung}.

For general conductors, Qin--Wu \cite{QW} reduce these cross terms to fourth moments of Kloosterman sums by primitive orthogonality and Poisson summation. We bound those moments uniformly in the prime-power factorisation. The table records earlier proportions.

\begin{center}
\small
\begin{tabular}{@{}p{6.1cm}p{5.2cm}p{3.1cm}@{}}
\hline
Work & Conductors & Proportion \\
\hline
Balasubramanian--Murty \cite{BM} (1992)
  & prime & $0.04$ \\
Iwaniec--Sarnak \cite{IS} (1999)
  & general & $1/3-o(1)$ \\
Michel--VanderKam \cite{MV} (2000)
  & general & $1/3-o(1)$ \\
Bui \cite{Bui} (2012)
  & general & $0.3411-o(1)$ \\
Khan--Ngo \cite{KhanNgo} (2016)
  & prime & $3/8-o(1)$ \\
Khan--Mili\'cevi\'c--Ngo \cite{KMN} (2022)
  & prime & $5/13-o(1)$ \\
Leung \cite{Leung} (2025)
  & $q^\eta$-smooth squarefree & $35.9\%$ \\
Qin--Wu \cite{QW} (2025)
  & general & $7/19-o(1)$ \\
\hline
\end{tabular}
\end{center}

The conductor restrictions prevent direct comparison of all rows. Our first theorem raises the general bound from $7/19-o(1)$ to $3/8-o(1)$, first reached for prime moduli by Khan--Ngo.

\begin{thm}\label{thm:allq-three-eighths}
For every $\ep>0$, every sufficiently large integer $q\not\equiv2\bmod4$ satisfies
$$
\kappa(q)\ge\frac38-\ep.
$$
For each $j\in\{0,1\}$ such that $\mathcal X^*(q;j)$ is nonempty,
$$
\kappa_j(q)\ge\frac38-\ep.
$$
\end{thm}

\looseness=-1 The proportions below depend on the repeated prime factors of $q$. Its powerful part is
\begin{equation}\label{eqn:powerful_part_definition}
P(q)=\prod_{\substack{\ell^\nu\Vert q\\\nu\ge2}}\ell^\nu,
\end{equation}
the product being empty, so $P(q)=1$, exactly when $q$ is squarefree.

\begin{thm}\label{thm:sqfree-nonvanishing}
For every $\ep>0$, all sufficiently large squarefree admissible moduli $q$ satisfy
$$
\kappa(q)\ge \frac5{13}-\ep.
$$
\end{thm}

Theorem~\ref{thm:sqfree-nonvanishing} is the case $P(q)=1$ of Theorem~\ref{thm:density-one-main}. Its proof in Section~\ref{sec:squarefree} is the model for the general argument.

\begin{thm}\label{thm:density-one-main}
For every $\ep>0$, every sufficiently large admissible modulus $q$ with
$$
P(q)\le q^{1/2}
$$
also satisfies
$$
\kappa(q)\ge\frac5{13}-\ep .
$$
Consequently, for all $X\ge2$,
$$
\#\{q\le X:q\not\equiv2\bmod4,\ \kappa(q)<5/13-\ep\}\ll_\ep X^{3/4},
$$
an exponent independent of $\ep$.
\end{thm}

Theorem~\ref{thm:density-one-main} applies to admissible $q=mr$ with $m$ fixed and $r$ squarefree, coprime to $m$ and sufficiently large, since $P(q)\le m\le q^{1/2}$. It also applies to $q=rp^2$ with $p$ an odd prime, $r$ odd and squarefree, $p\nmid r$ and $r\ge p^2$, since $P(q)=p^2\le q^{1/2}$. Together with Theorem~\ref{thm:allq-three-eighths}, it shows that every sufficiently large admissible modulus $q$ satisfies $\kappa(q)\ge5/13-\ep$ if $P(q)\le q^{1/2}$, and $\kappa(q)\ge3/8-\ep$ otherwise.

Proposition~\ref{prop:reciprocal-factorable} is strongest when the chosen divisor of $P(q)$ has size $q^{1/3}$. This gives a proportion exceeding $5/13$ for the following family.

\begin{thm}\label{thm:two-fifths}
For every $\ep>0$ there is $\eta>0$ such that all sufficiently large moduli $q=rp^2$, where $p$ is an odd prime, $r$ is odd and squarefree, $p\nmid r$ and
$$
p^{4-\eta}\le r\le p^{4+\eta},
$$
satisfy
$$
\kappa(rp^2)\ge\frac25-\ep .
$$
\end{thm}

This is the largest proportion in the paper, and there is no exceptional set. This follows from Corollary~\ref{cor:third-divisor}, since $P(rp^2)=p^2=q^{1/3+O(\eta)}$. The remark preceding that corollary explains the optimal size $q^{1/3}$ and the limitation to $3/8$ for general $q$.

For prime-square conductors we obtain a power-saving exceptional set of primes.

\begin{thm}\label{thm:p2-nonvanishing}
For every $\ep>0$ there are $\xi_1>0$ and a set of primes $\mathcal E$ satisfying
$$
\#\{p\le X:p\in\mathcal E\}
\ll_\ep \frac{X^{1-\xi_1}}{\log X},
$$
such that every sufficiently large prime $p\notin\mathcal E$ satisfies
$$
\kappa(p^2)\ge \frac5{13}-\ep.
$$
\end{thm}

The prime-square argument also applies after adjoining a fixed coprime factor.

\begin{thm}\label{thm:fixed-m-main}
Let $m\ge1$ be fixed, with $m\not\equiv2\bmod4$. For every $\ep>0$ there are $\xi_1>0$ and a set of primes $\mathcal E$ satisfying
$$
\#\{p\le X:p\in\mathcal E\}\ll_\ep\frac{X^{1-\xi_1}}{\log X},
$$
such that every sufficiently large prime $p\notin\mathcal E$ satisfies
$$
\kappa(mp^2)\ge \frac{5}{13}-\ep.
$$
\end{thm}

The cofactor may instead vary through a range of squarefree integers.

\begin{thm}\label{thm:hybrid-interpolation}
Fix $0\le a<1$ and let $\ep>0$. There are $\xi_1>0$ and a set of primes $\mathcal E$ such that
$$
\#\{p\le X:p\in\mathcal E\}\ll_{a,\ep}\frac{X^{1-\xi_1}}{\log X},
$$
and, for every sufficiently large odd prime $p\notin\mathcal E$ and every odd squarefree $r\le p^a$ with $p\nmid r$,
$$
\kappa(rp^2)\ge \frac{2a+5}{6a+13}-\ep.
$$
\end{thm}

As $a\to0$ the proportion tends to $5/13$, since
$$
\frac5{13}-\frac{2a+5}{6a+13}=\frac{4a}{13(6a+13)},
$$
so choosing $a>0$ with $4a/(13(6a+13))<\ep/2$ and applying Theorem~\ref{thm:hybrid-interpolation} with error $\ep/2$ gives $\kappa(rp^2)\ge5/13-\ep$ throughout its stated range.

The squarefree and prime-square theorems rest on fourth-moment estimates. Throughout we use the notation $e(x)=\exp(2\pi ix)$ and
$$
S(a,b;c)=\sum_{\substack{x\bmod c\\(x,c)=1}}
e\left(\frac{ax+b\bar{x}}{c}\right),
$$
where a bar denotes inversion in the relevant modulus. For $B\ge1$, denote
$$
A_q(B)=
\sum_{\substack{1\le b_1,b_2,b_3,b_4\le B\\
(b_1b_2b_3b_4,q)=1}}
\left|
\sum_{h\bmod q}\prod_{i=1}^4S(h,\overline{b_i};q)
\right|.
$$
These complete correlations arise on expanding the fourth power in the mollifier calculation. We prove the following bounds.
\begin{thm}\label{thm:sqfree-average}
For every $\ep>0$, every squarefree $q$, and every $B\ge1$,
$$
A_q(B)\ll_\ep q^\ep(B^4q^{5/2}+B^2q^3).
$$
\end{thm}

\begin{thm}\label{thm:p2-average}
For all $\ep,\delta>0$ and $0<\xi<\min(\ep,1)$ there is a set of primes $\mathcal E$ satisfying
$$
\#\{p\le X:p\in\mathcal E\}\ll_{\ep,\delta,\xi}\frac{X^{1-\xi}}{\log X}
$$
such that, for every odd prime $p\notin\mathcal E$ and uniformly for $1\le B\le p^{1/2-\delta}$,
$$
A_{p^2}(B)\ll_{\ep,\delta,\xi}p^\ep(B^4p^5+B^2p^6).
$$
The set $\mathcal E$ does not depend on $B$.
\end{thm}

Several related problems use the same mollifiers or the same estimates for Kloosterman sums. \v{C}ech and Matom\"aki \cite{CechMatOptimal} show that the Michel--VanderKam mollifier is optimal in a wide class of balanced two-piece mollifiers. Bui, Pratt, Robles and Zaharescu \cite{BPRZ} treat twisted second moments at prime moduli, and mollifiers give positive proportions of zeros on the critical line for the Riemann zeta function \cite{Conrey} and for Dirichlet $L$-functions \cite{WuTwisted}. Fourth moments of Dirichlet $L$-functions are obtained at the central point in \cite{Young,WuCentral}, and on average over height and primitive characters in \cite{WuLine}.

Bui, Pratt and Zaharescu \cite{BPZ} prove non-vanishing results under the assumption that exceptional characters exist, and Drappeau, Pratt and Radziwi\l\l{} \cite{DPR} prove one-level density estimates on average over conductors, with consequences for non-vanishing under the generalised Riemann hypothesis. For modular $L$-functions see \cite{ISAuto,KMV}, and for related non-vanishing problems see \cite{Castillo, David, Diaconu, Sound}. Incomplete and bilinear Kloosterman sums are treated in \cite{FI,KMS,KSWX}.
\section{Overview of the proofs}\label{sec:outline}

\looseness=-1 We use the Michel--VanderKam two-piece mollifier and Qin--Wu's reduction of its mixed terms to Kloosterman sums. For admissible exponents $\theta_1,\theta_2$, Corollary \ref{cor:qw-nv} gives
$$
\kappa(q)\ge\frac{\theta_1+\theta_2}{1+\theta_1+\theta_2}+o(1).
$$
\looseness=-1 We therefore maximise $\theta_1+\theta_2$. Section \ref{sec:qw-reduction} gives the common moment calculation. For squarefree moduli, local correlations have square-root cancellation unless the four arguments form equal pairs; Section \ref{sec:squarefree} counts these congruences. For prime squares, the local evaluation gives a polynomial relation among reciprocal square roots. Section \ref{sec:p2-kloosterman} counts its integer zeros and averages its non-zero values over primes, with an exceptional set independent of the length. Both cases allow $\theta_1$ and $\theta_2$ just below $3/8$ and $1/4$, giving $5/13$.

Section \ref{sec:small-nsf} treats fixed multiples and the powerful part. Counting root configurations, rather than bounding each correlation separately, permits the squarefree lengths when $P(q)\le q^{1/2}$; the sparsity of powerful integers gives the density-one theorem.

For the uniform $3/8$ theorem, Section \ref{sec:allq-three-eighths} treats a small powerful part by the preceding fourth moment, an intermediate divisor by Proposition~\ref{prop:reciprocal-factorable}, and a large prime square by the operator bound of Proposition~\ref{prop:p2-crt-operator} and the Chinese remainder theorem. The lengths are just below $2/5$ and $1/5$; a divisor near $q^{1/3}$ gives $2/5$. Section \ref{sec:hybrid} makes the prime-square estimates simultaneous in the squarefree cofactor.

\section{Preliminaries}\label{sec:preliminaries}

We retain the notation of the Introduction. Write $\omega(n)$ for the number of distinct prime divisors and $\sum_{x\bmod c}^{\times}$ for summation over units. At modulus one, this sum has one term, $S(a,b;1)=1$, $A_1(B)\ll B^4$ and $\omega(1)=0$. Changes of variables give, for $(u,c)=1$ and $(r,s)=1$,
\begin{gather*}
S(a,b;c)=S(au,b\bar{u};c)=S(a\bar{u},bu;c),\displaybreak[0]\\
\overline{S(a,b;c)}=S(-a,-b;c)=S(a,b;c),\displaybreak[0]\\
S(a,b;rs)=S(a\bar{s},b\bar{s};r)S(a\bar{r},b\bar{r};s).
\end{gather*}
In particular, the sums are real. For an odd prime $p$, write $\operatorname{Kl}_2(t;p)=p^{-1/2}S(1,t;p)$.

\begin{lem}\label{lem:phistar-loss}
For every $\ep>0$ and every $q\not\equiv2\bmod4$ with $\varphi^*(q)>0$,
$$
\frac{q}{\varphi^*(q)}\ll_\ep q^\ep.
$$
\end{lem}

\begin{proof}
The function $\varphi^*(q)$ is multiplicative. For odd primes,
$$
\varphi^*(p)=p-2,\qquad \varphi^*(p^a)=p^{a-2}(p-1)^2\quad(a\ge2).
$$
 The allowed powers of $2$ satisfy the same lower bound up to an absolute constant. Hence
$$
\frac{q}{\varphi^*(q)}
\ll 2^{\omega(q)}\prod_{p\mid q}(1-1/p)^{-2}
\ll_\ep q^\ep.
$$
\end{proof}

\begin{lem}\label{lem:parity-count}
For $q\not\equiv2\bmod4$ and $j\in\{0,1\}$,
$$
\#\mathcal X^*(q;j)=\frac{\varphi^*(q)}2+O(1).
$$
More precisely,
$$
\#\mathcal X^*(q;j)
=\frac12\left(\varphi^*(q)+(-1)^j
\left(\mu(q)+1_{2\mid q}\mu(q/2)\right)\right).
$$
\end{lem}

\begin{proof}
The parity projector gives
$$
\#\mathcal X^*(q;j)=\frac12\Bigl(\varphi^*(q)+(-1)^j\sum_{\chi\bmod q}^{*}\chi(-1)\Bigr).
$$
Primitive orthogonality gives $\sum_{d\mid q,\ d\mid 2}\mu(q/d)\varphi(d)$; only $d=1$ and, for even $q$, $d=2$ occur, with $\varphi(1)=\varphi(2)=1$.
\end{proof}

\begin{prop}\label{prop:fgkm-prime-local}
Let $p$ be an odd prime, and let $a_1,a_2,a_3,a_4\in\F_p^\times$. Set
$$
C(a_1,a_2,a_3,a_4)=
\sum_{h\bmod p}
\prod_{i=1}^4 \operatorname{Kl}_2(ha_i;p).
$$
Then
$$
C(a_1,a_2,a_3,a_4)\ll p.
$$
Unless the four residues form two equal pairs, the sharper bound
$$
C(a_1,a_2,a_3,a_4)\ll p^{1/2}
$$
holds. Thus if one residue occurs only once, then
$$
\sum_{h\bmod p}
S(h,a_1;p)S(h,a_2;p)S(h,a_3;p)S(h,a_4;p)
\ll p^{5/2}.
$$
In the remaining case the same unnormalised sum is $O(p^3)$. The constants are absolute.
\end{prop}

\begin{proof}
Apply \cite[Proposition 3.2]{FGKM} with $\kappa=4$ and
$$
\gamma_i=\begin{pmatrix}a_i&0\\0&1\end{pmatrix}.
$$
These pullbacks have bounded conductor on $\F_p^\times$, and equality in ${\rm PGL}_2$ is equivalent to equality of the $a_i$. The exceptional case is therefore when the four arguments can be paired into equal pairs. The convention $\operatorname{Kl}_{\rm FGKM}(t;p)=p^{-1/2}S(t,1;p)$ agrees with ours by $x\mapsto\bar x$. Moreover
$$
S(h,a_i;p)=p^{1/2}\operatorname{Kl}_2(ha_i;p),
$$
including $h=0$, when both sides are $-1$. Adding this term changes the normalised correlation by $p^{-2}$. Multiplying the cited bounds by $p^2$ gives the result.
\end{proof}

For background on this correlation estimate, see also \cite[Theorem~13.3]{FKMS}.

\begin{prop}\label{prop:weil-estermann}
For every modulus $n\ge1$,
$$
|S(a,b;n)|\le \tau(n)(a,b,n)^{1/2}n^{1/2}.
$$
Moreover, if $n$ is powerful and $(b,n)=1$, then
$$
S(h,b;n)=0
$$
unless $(h,n)=1$. Consequently, if $(ab,n)=1$, then
$$
|S(a,b;n)|\ll_\ep n^{1/2+\ep},
$$
and if $(b_1b_2b_3b_4,n)=1$ and $n$ is powerful, then
$$
\left|
\sum_{h\bmod n}\prod_{i=1}^4 S(h,\overline{b_i};n)
\right|
\ll_\ep n^{3+\ep}.
$$
\end{prop}

\begin{proof}
The first estimate is the Weil--Estermann bound \cite[Corollary 11.12, (11.16)]{IK}. If $p^\nu\Vert n$, $\nu\ge2$, $p\mid h$ and $p\nmid b$, write a unit as $y=y_0+p^{\nu-1}t$, $t\bmod p$. Since
$$
(y_0+p^{\nu-1}t)^{-1}\equiv \bar y_0-p^{\nu-1}t\bar y_0^2\pmod{p^\nu},
$$
the $t$-sum is a nontrivial additive sum with coefficient $-b\bar y_0^2\bmod p$, and vanishes. Multiplicativity proves the vanishing assertion. For unit $h,b$, Weil--Estermann gives $|S(h,b;n)|\ll_\ep n^{1/2+\ep}$; summing the product of four such bounds over at most $n$ values of $h$ gives $n^{3+\ep}$ after renaming $\ep$.
\end{proof}

The following prime-power evaluations are used in Sections~\ref{sec:p2-kloosterman}, \ref{sec:small-nsf} and \ref{sec:allq-three-eighths}.

For an odd prime $p$, a unit $A$, and $\nu\ge2$, let
\begin{equation}\label{eq:reciprocal-gauss-sign}
 \tau_p(A;p^\nu)=
 \begin{cases}
  1,&2\mid\nu,\\
  \left(\dfrac Ap\right),&2\nmid\nu,\ p\equiv1\bmod4,\\
  i\left(\dfrac Ap\right),&2\nmid\nu,\ p\equiv3\bmod4.
 \end{cases}
\end{equation}
Let $p\nmid m$. For every non-zero square modulo $p$, choose one of its two square roots. Hensel lifting gives a compatible square-root function $z\mapsto z_{1/2}$ on the square units modulo every $p^\nu$. Thus the chosen root modulo $p^\nu$ is determined by its reduction modulo $p$. For $am$ a square unit, define
\begin{equation}\label{eq:reciprocal-odd-branch}
 S_{p^\nu}^{\epsilon}(a,m)
 =p^{\nu/2}\tau_p
   \left(\epsilon(am)_{1/2};p^\nu\right)
  \e\left(\frac{2\epsilon(am)_{1/2}}{p^\nu}\right),
 \qquad \epsilon\in\{+1,-1\},
\end{equation}
and set it equal to zero otherwise. The stationary-phase evaluation of Kloosterman sums modulo odd prime powers gives
\begin{equation}\label{eq:reciprocal-odd-branch-sum}
 S(a,m;p^\nu)=S_{p^\nu}^{+}(a,m)+S_{p^\nu}^{-}(a,m).
\end{equation}
This follows from \cite[Lemmas~12.2 and~12.3]{IK} with trivial character and the quadratic Gauss-sum evaluation; see also \cite[Lemma~22]{BM2015}. The root choice only labels the two terms in \eqref{eq:reciprocal-odd-branch-sum}.

For $2^\nu$ with $\nu=6$ or $\nu\ge8$ and $am\equiv1\bmod8$, let $(am)_{1/2}$ be the unique root class modulo $2^{\nu-1}$ which is congruent to $1\bmod4$ and whose square is $am\bmod {2^\nu}$. These choices are compatible under reduction. Define the two-primary Gauss factor by
\begin{equation}\label{eq:reciprocal-two-gauss-sign}
 \tau_2(A;2^\nu)=
 \begin{cases}
  \{1+\e(A/4)\}/\sqrt2,&2\mid\nu,\\
  \e(A/8),&2\nmid\nu.
 \end{cases}
\end{equation}
For odd $a,m$, define
\begin{equation}\label{eq:reciprocal-two-branch}
 S_{2^\nu}^{\epsilon}(a,m)=
 \begin{cases}
  2^{(\nu+1)/2}\tau_2
   \left(\epsilon(am)_{1/2};2^\nu\right)
   \e\left(\dfrac{2\epsilon(am)_{1/2}}{2^\nu}\right),
       &a\equiv m\bmod8,\\
  0,&a\not\equiv m\bmod8.
 \end{cases}
\end{equation}
Again the sum over $\epsilon$ is $S(a,m;2^\nu)$; see \cite[Section~3.1]{MQW}. At $\nu\in\{2,3,4,5,7\}$ we retain the full factor $S(a,m;2^\nu)$ and bound it directly; these exceptional moduli are bounded.

\section{Reduction of the mollifier cross terms to quartic Kloosterman sums}\label{sec:qw-reduction}

For $0<\theta<1$ define
$$
a_\theta(m)=
\begin{cases}
\mu(m)\left(1-\frac{\log m}{\theta\log q}\right),
&1\le m\le q^\theta,\\
0,&m>q^\theta.
\end{cases}
$$
For primitive $\chi$ of parity $j$, write $\tau_\chi=\sum_{a\bmod q}\chi(a)\e(a/q)$ and set $\epsilon_\chi=\tau_\chi/(i^j\sqrt q)$, so $|\epsilon_\chi|=1$, and define
\begin{equation}\label{eq:two-piece-mollifier}
M^{(j)}(\chi)
=
c_1\sum_{m\le q^{\theta_1}}\frac{a_{\theta_1}(m)\chi(m)}{\sqrt m}
+
c_2\overline{\epsilon_\chi}
\sum_{m\le q^{\theta_2}}\frac{a_{\theta_2}(m)\overline\chi(m)}{\sqrt m}.
\end{equation}
Define the normalised moments
$$
\mathcal M_1^{(j)}(q)
=
\frac{2}{\varphi^*(q)}
\sum_{\chi\in\mathcal X^*(q;j)}
L(1/2,\chi)M^{(j)}(\chi)
$$
and
$$
\mathcal M_2^{(j)}(q)
=
\frac{2}{\varphi^*(q)}
\sum_{\chi\in\mathcal X^*(q;j)}
|L(1/2,\chi)M^{(j)}(\chi)|^2.
$$

For $j\in\{0,1\}$ set
$$
V_j(x)=\frac{1}{2\pi i}\int_{(2)}
\frac{\Gamma(s/2+(2j+1)/4)^2}
{\Gamma((2j+1)/4)^2}
(\pi x)^{-s}\frac{ds}{s}.
$$
The weight $V_0(x)$ is the one used in Qin--Wu \cite[(3.2)]{QW}; $V_1(x)$ is its odd-parity analogue. For a primitive character $\chi\bmod q$ of parity $j$,
$$
|L(1/2,\chi)|^2
=
2\sum_{n_1,n_2\ge1}
\frac{\chi(n_1)\overline\chi(n_2)}{(n_1n_2)^{1/2}}
V_j\left(\frac{n_1n_2}{q}\right).
$$
For every $A,n\ge0$ and both $j\in\{0,1\}$,
\begin{equation}\label{eq:V-derivative-bounds}
x^n V_j^{(n)}(x)\ll_{A,n,j} (1+x)^{-A},
\end{equation}
and the Mellin transform has residue $1$ at $s=0$. More precisely,
$$
V_j(x)=\frac{4\pi^{(2j+1)/2}}{\Gamma((2j+1)/4)^2}
\int_x^\infty t^{(2j-1)/2}K_0(2\pi t)\,dt,
$$
where $K_0$ is the modified Bessel function. Its exponential decay permits truncation at $n_1n_2\le Cq\log q$, with error $O(q^{-100})$ for a sufficiently large fixed $C$.

After dyadic subdivision,
$$
M_1\le q^{\theta_1+o(1)},\qquad
M_2\le q^{\theta_2+o(1)},\qquad
N_1N_2\ll q^{1+\varepsilon_1}.
$$
\begin{lem}\label{lem:primitive-factorisation}
Let $(m,q)=1$ and $j\in\{0,1\}$. Then
$$
\sum_{\chi\in\mathcal X^*(q;j)}\tau_\chi\chi(m)
=
\frac12\sum_{\substack{q_1q_2=q\\ (q_1,q_2)=1}}
\mu^2(q_1)\varphi(q_2)
\left(\e\left({\frac{\overline{mq_1}}{q_2}}\right)
+(-1)^j
\e\left({-\frac{\overline{mq_1}}{q_2}}\right)\right).
$$
Consequently every non-zero term has $q=q_1q_2$, $(q_1,q_2)=1$, and $\mu^2(q_1)=1$. If $\ell^\nu\Vert q$ with $\nu\ge2$, the full prime power $\ell^\nu$ lies in $q_2$.
\end{lem}

\begin{proof}
Primitive orthogonality, followed by expansion of the Gauss sum, gives
$$
\sum_{\chi\bmod q}^{*}\tau_\chi\chi(m)
=
\sum_{q_1q_2=q}\mu(q_1)\varphi(q_2)
\sum_{\substack{a\bmod q\\(a,q)=1\\ am\equiv1\bmod {q_2}}}
\e(a/q).
$$
If $(q_1,q_2)>1$, summing the lifts at a common prime makes the inner sum zero. Indeed, if a prime $p$ divides $q_1$ and $q_2$, write $a=a_0+q_2t$ with $t\bmod q_1$. The shift $t\mapsto t+q_1/p$ preserves the conditions on $a$ and multiplies $\e(a/q)$ by $\e(1/p)\ne1$. Otherwise CRT evaluates it as $\e(\overline{mq_1}/q_2)c_{q_1}(\bar q_2)$; since the surviving $q_1$ are squarefree, $c_{q_1}(\bar q_2)=\mu(q_1)$.

This gives the unprojected formula underlying \cite[(3.3)]{QW} and \cite[(17)]{MV}. The parity projector takes half its value at $m$ plus $(-1)^j/2$ times its value at $-m$. Squarefreeness of $q_1$ and $(q_1,q_2)=1$ place every repeated prime power in $q_2$. The modulus-one convention includes $q_2=1$.
\end{proof}

\begin{lem}\label{lem:parity-orthogonality}
Let $(a b,q)=1$. Then
$$
\sum_{\chi\in\mathcal X^*(q;j)}\chi(a)\overline{\chi(b)}
=
\frac12\sum_{\chi\bmod q}^{*}\chi(a\bar{b})
+
\frac{(-1)^j}{2}\sum_{\chi\bmod q}^{*}\chi(-a\bar{b}).
$$
Equivalently,
$$
\sum_{\chi\in\mathcal X^*(q;j)}\chi(a)\overline{\chi(b)}
=
\frac12\sum_{\substack{d\mid q\\ a\equiv b\bmod d}}\mu\left(\frac{q}{d}\right)\varphi(d)
+
\frac{(-1)^j}{2}
\sum_{\substack{d\mid q\\ a\equiv -b\bmod d}}\mu\left(\frac{q}{d}\right)\varphi(d).
$$
\end{lem}

\begin{proof}
Insert $(1+(-1)^j\chi(-1))/2$ and use
$$
\sum_{\chi\bmod q}^{*}\chi(u)=\sum_{\substack{d\mid q\\ u\equiv1\bmod d}}\mu(q/d)\varphi(d).
$$

\end{proof}

\begin{lem}\label{lem:square-term-parity-transfer}
Fix $0<\theta<1/2$. For $j\in\{0,1\}$, set
$$
P_\theta(\chi)=
\sum_{m\le q^\theta}\frac{a_\theta(m)\chi(m)}{\sqrt m}.
$$
Then, uniformly for $q>1$ with $q\not\equiv2\bmod4$ and for both parities,
$$
\frac{2}{\varphi^*(q)}
\sum_{\chi\in\mathcal X^*(q;j)}
|L(1/2,\chi)|^2|P_\theta(\chi)|^2
=1+\frac1\theta
+O_{\theta}\left(\frac{\log\log q}{\log q}\right).
$$
\end{lem}

\begin{proof}
The approximate functional equation and Lemma \ref{lem:parity-orthogonality} give
\begin{equation}\label{eq:square-parity-divisor-expansion}
\begin{aligned}
&\frac{2}{\varphi^*(q)}
\sum_{\chi\in\mathcal X^*(q;j)}
|L(1/2,\chi)|^2|P_\theta(\chi)|^2\\
&=
\frac{2}{\varphi^*(q)}
\sum_{\delta=\pm1}(-1)^{j(1-\delta)/2}
\sum_{d\mid q}\mu(q/d)\varphi(d)
\sum_{\substack{m,m'\le q^\theta,\ n_1,n_2\ge1\\
(mm'n_1n_2,q)=1\\ n_1m\equiv\delta n_2m'\bmod d}}
\frac{a_\theta(m)a_\theta(m')}
{(mm'n_1n_2)^{1/2}}
V_j(n_1n_2/q).
\end{aligned}
\end{equation}
For $j=0$, these congruence sums are evaluated immediately before and in \cite[(3.1), p.~7]{QW}; the odd moment is stated in \cite[p.~5]{QW}. For $j=1$, the gamma quotient is $\Gamma(s/2+3/4)^2/\Gamma(3/4)^2$, with value $1$ at the origin and the same Stirling bounds. The signs above are exactly those of the parity projector, so these evaluations give the assertion.
\end{proof}

In the mixed term formed from the first mollifier piece and the conjugate of the second, the root-number and Gauss-sum factors give
$$
2c_1c_2 q^{-1/2}\mu^2(q_1)\varphi(q_2)
\frac12\Re\left((-i)^j
\left(\e(t)+(-1)^j\e(-t)\right)\right),
\qquad
t=\frac{n_2\overline{n_1m_1m_2q_1}}{q_2}.
$$
The last factor reduces to
$$
\cos(2\pi t)\quad(j=0),
\qquad
\sin(2\pi t)\quad(j=1).
$$
For $j=0$, this is the phase in Qin--Wu \cite[(3.2)--(3.4)]{QW}. The term $m_1m_2n_1=1$ is evaluated after \cite[(3.4)]{QW} as in Michel--VanderKam \cite[Section 6.1]{MV}, and contributes $2c_1c_2+O(q^{-\eta})$.

For $j=1$, we verify the diagonal in the odd moment stated before Qin--Wu's Proposition 2.1 \cite[p.~5]{QW} directly. Write
$$
D(q)=
\sum_{\substack{q_1q_2=q\\(q_1,q_2)=1}}
\mu^2(q_1)\varphi(q_2)
\sum_{\substack{n\ge1\\(n,q)=1}}
\frac{V_1(n/q)}{\sqrt n}
\sin\left(\frac{2\pi n\overline{q_1}}{q_2}\right).
$$
The $q_2=1$ summand vanishes. The two conjugate mixed diagonals contribute
$$
\frac{8c_1c_2D(q)}{\varphi^*(q)\sqrt q}.
$$
Lemma~\ref{lem:primitive-factorisation} and the functional equation give
$$
\frac{D(q)}{\sqrt q}
=\frac1{2\pi i}\int_{(1)}
\frac{\Gamma(3/4+s/2)\Gamma(3/4-s/2)}{\Gamma(3/4)^2}
\sum_{\chi\in\mathcal X^*(q;1)}L(1/2-s,\chi)\frac{ds}{s}.
$$
Shift to $\Re s=-1$. Only the pole at $s=0$ is crossed. The gamma quotient is even and has value $1$ at zero, so its integral divided by $s$ on $\Re s=-1$ is $-1/2$. Expanding the absolutely convergent $L$-series on this line and applying primitive orthogonality to $n\ge2$ gives an error bounded by
$$
\sum_{d\mid q}\varphi(d)
\sum_{\substack{n\ge2\\n\equiv\pm1\bmod d}}n^{-3/2}
\ll_\varepsilon q^\varepsilon,
$$
since the inner sum is $O(d^{-3/2})$ for $d\ge3$ and $O(1)$ for $d=1,2$. Thus
$$
\frac{D(q)}{\sqrt q}
=\sum_{\chi\in\mathcal X^*(q;1)}L(1/2,\chi)
-\frac12\#\mathcal X^*(q;1)+O_\varepsilon(q^\varepsilon).
$$
The single-$L$ approximate functional equation gives
$$
\sum_{\chi\in\mathcal X^*(q;1)}L(1/2,\chi)
=\#\mathcal X^*(q;1)+O_\varepsilon(q^{3/4+\varepsilon}).
$$
Indeed, take its weight
$$
U(x)=\frac{1}{\Gamma(3/4)}\int_{\pi x^2}^{\infty}e^{-t}t^{-1/4}\,dt.
$$
Both sums may be truncated at $n\le q^{1/2+\varepsilon}$. In the first sum, the term $n=1$ has weight $1+O(q^{-3/4})$. Primitive orthogonality bounds the other terms by $O_\varepsilon(q^{1/4+\varepsilon})$, using
$$
\sum_{\substack{2\le n\le N\\ n\equiv\pm1\bmod d}}n^{-1/2}\ll\frac{\sqrt N}{d}.
$$
In the dual sum, Lemma~\ref{lem:primitive-factorisation} gives
$$
\biggl|\sum_{\chi\in\mathcal X^*(q;1)}\epsilon_\chi\overline\chi(n)\biggr|
\le q^{-1/2}\sum_{q_1q_2=q}\varphi(q_2)\ll q^{1/2},
$$
so the dual sum is $O_\varepsilon(q^{3/4+\varepsilon})$. Consequently Lemma~\ref{lem:parity-count} gives
$$
D(q)=\frac14\varphi^*(q)\sqrt q
+O_\varepsilon(q^{5/4+\varepsilon}),
$$
and Lemma~\ref{lem:phistar-loss} gives the mixed main term $2c_1c_2+O(q^{-\eta})$ for some $\eta>0$.

Thus the two conjugate mixed terms contribute $2c_1c_2$ in either parity.

\subsection{Parity-separated first and second moments}
\label{sec:fixed-parity-qw-reduction}

We isolate the other mixed terms.

\begin{prop}\label{prop:fixed-parity-qw-setup}
Fix $\sigma>0$, lengths $0<\theta_1,\theta_2<1/2$, and real numbers $c_1,c_2$. Choose $\varepsilon_1>0$ sufficiently small in terms of $\sigma$. Let $q>1$ satisfy $q\not\equiv2\bmod4$, and let $j\in\{0,1\}$. Let the dyadic weights be supported in a fixed compact subinterval of $(0,\infty)$ and have sufficiently many uniformly bounded derivatives. Assume
$$
M_1\le q^{\theta_1+o(1)},\qquad
M_2\le q^{\theta_2+o(1)},\qquad
N_1N_2\le q^{1+o(1)},
$$
where the three $o(1)$-terms tend to zero uniformly and are bounded in absolute value by a sufficiently small constant depending on $\theta_1,\theta_2,\sigma$. Then
$$
\mathcal M_1^{(j)}(q)=c_1+c_2+O(q^{-\eta})
$$
for some $\eta>0$, depending only on $\sigma,\theta_1,\theta_2$ and the prescribed support and derivative bounds of the weights, but not on $q$; the implied constant may also depend on $c_1,c_2$.

Moreover,
$$
\begin{aligned}
\mathcal M_2^{(j)}(q)
&=
c_1^2\left(1+\frac{1}{\theta_1}\right)
+c_2^2\left(1+\frac{1}{\theta_2}\right)
+2c_1c_2+\mathcal R^{(j)}(q)
+O_{\theta_1,\theta_2}\left(\frac{\log\log q}{\log q}\right).
\end{aligned}
$$
Here $2c_1c_2$ is the contribution from $m_1m_2n_1=1$ in the two cross terms. Up to $O(q^{-\eta})$, $\mathcal R^{(j)}(q)$ is the sum over both signs, all coprime factorisations $q=q_1q_2$ with $\mu^2(q_1)=1$, and the finite dyadic subdivision. Each summand is a coefficient $O_{c_1,c_2}(1)$ times
$$
\begin{aligned}
&\frac{\mu^2(q_1)\varphi(q_2)}
{\varphi^*(q)q^{1/2}(M_1M_2N_1N_2)^{1/2}}
\sum_{\substack{m_1,m_2,n_1,n_2\\ (m_1m_2n_1n_2,q)=1\\ m_1m_2n_1>1}}
\alpha_{m_1}\beta_{m_2}
W_1\left(\frac{n_1}{N_1}\right)W_2\left(\frac{n_2}{N_2}\right)V_j(n_1n_2/q)\\
&\qquad\times\e\left(\pm \frac{n_2\overline{n_1m_1m_2q_1}}{q_2}\right).
\end{aligned}
$$
The coefficients are divisor-bounded, supported on $m_1\asymp M_1$ and $m_2\asymp M_2$, and the variables are coprime to $q$.
\end{prop}

\begin{proof}
For the first piece, \cite[Proposition 2.1 and (2.4)]{QW}, together with the odd-parity statement preceding that proposition, gives the first moment in our normalisation. At the central point, $L(1/2,\chi)\overline{\epsilon_\chi}=L(1/2,\overline\chi)$; conjugation preserves parity, so the second piece has the same first moment. Lemma \ref{lem:square-term-parity-transfer} evaluates both self-products, since $|\bar\epsilon_\chi P_\theta(\bar\chi)|=|P_\theta(\chi)|$.

For the mixed products, \cite[(3.2)--(3.4)]{QW} and Lemma \ref{lem:primitive-factorisation} give the phases and coefficients above; the mixed diagonal contributes $2c_1c_2$ in either parity. Dyadic subdivision gives the remaining blocks, with $N_1N_2\le q^{1+o(1)}$ by decay of $V_j(x)$. The factorisations have squarefree $q_1$ and include $q_2=1$.
\end{proof}

\subsection{Poisson summation in the remaining cross terms}

Fix a factorisation ${q=q_1q_2}$ with $(q_1,q_2)=1$ and $\mu^2(q_1)=1$. For each sign, define
$$
\begin{aligned}
\mathcal B_\pm(q_1,q_2)
&=\frac{\mu^2(q_1)\varphi(q_2)}{q^{3/2}(M_1M_2N_1N_2)^{1/2}}\\
&\quad\times
\sum_{\substack{m_1,m_2\\ (m_1m_2,q)=1}}
\alpha_{m_1}\beta_{m_2}
\sum_{\substack{n_1,n_2\\ (n_1n_2,q)=1\\ m_1m_2n_1>1}}
V_j\left(\frac{n_1n_2}{q}\right)\\
&\quad\times
W_1\left(\frac{n_1}{N_1}\right)
W_2\left(\frac{n_2}{N_2}\right)
\e\left(\pm \frac{n_2\overline{n_1m_1m_2q_1}}{q_2}\right).
\end{aligned}
$$
The coefficients are divisor-bounded on the indicated intervals. Inverses are modulo $q_2$, with the modulus-one convention. Replacing $\varphi^*(q)^{-1}q^{-1/2}$ in Proposition \ref{prop:fixed-parity-qw-setup} by $q^{-3/2}$ costs $q/\varphi^*(q)$, absorbed by Lemma \ref{lem:phistar-loss}.

\begin{lem}\label{lem:fixed-factorisation-large-ratio}
For either sign, set
$$
\mathcal B_\pm^{\rm sum}(q)=
\sum_{\substack{q_1q_2=q\\ (q_1,q_2)=1\\ \mu^2(q_1)=1}}
\mathcal B_\pm(q_1,q_2).
$$
\looseness=-1 The weights are smooth and satisfy the derivative bounds in Proposition \ref{prop:fixed-parity-qw-setup}, and the coefficients are divisor-bounded and supported on $m_1\asymp M_1$ and $m_2\asymp M_2$. Fix $\eta_1,\sigma>0$. If
$$
M_1M_2\le q^{1-\eta_1}
$$
and
$$
\frac{N_2}{N_1}\ge q^\sigma\frac{M_1M_2}{q},
$$
then
$$
\mathcal B_\pm^{\rm sum}(q)
\ll_{\varepsilon_1}
q^{\varepsilon_1}\left(\frac{M_1M_2N_1}{qN_2}\right)^{1/2}+q^{-\eta},
$$
where $\eta>0$ depends only on $\eta_1,\sigma$ and the prescribed support and derivative bounds of the weights. The estimate concerns the sum over the coprime factorisations $q=q_1q_2$.
\end{lem}

\begin{proof}
For the positive phase this is \cite[Lemma 3.1, with $\mathcal B$ defined on p.~8]{QW}. Their block includes the sum over coprime factorisations, with coefficient $\mu^2(q_1)\varphi(q_2)$ and normalisation $q^{-3/2}(M_1M_2N_1N_2)^{-1/2}$. Its hypotheses, which include $M_1M_2<q^{1-\ep}$ and $N_1N_2\le q^{1+\ep}$, follow by taking the dyadic losses small relative to $\eta_1,\sigma$. For real smooth weights, the negative phase is the conjugate of the positive phase with conjugated coefficients. The derivative bounds \eqref{eq:V-derivative-bounds} treat both parities. Finally $M_1M_2N_1/(qN_2)\le q^{-\sigma}$, so $\varepsilon_1<\sigma/4$ leaves a power saving.
\end{proof}

\begin{lem}\label{lem:corrected-complete-sum}
Let $q_2\ge1$, let $(rm_1m_2n_2q_1,q_2)=1$, and let $k\in\mathbb Z$, with no restriction on $(k,q_2)$. With the convention $S(a,b;1)=1$, one has, for each sign,
$$
\sum_{x\bmod q_2}^{\times}
\e\left(\frac{kx}{q_2}\pm \frac{n_2\overline{rxm_1m_2q_1}}{q_2}\right)
=
S(\pm kn_2\overline{rm_1q_1},\overline{m_2};q_2).
$$
\end{lem}

\begin{proof}
At $q_2=1$ both sides are $1$. Otherwise apply the scaling identity to $S(k,\pm n_2\overline{rm_1m_2q_1};q_2)$ with the units $u=\pm n_2$ and $u=\overline{rm_1q_1}$; $k$ need not be a unit.
\end{proof}

\begin{prop}\label{prop:qw-preinsertion}
Fix $\eta_1,\eta_2>0$. Let $M_1,M_2,N_1,N_2$ be one off-diagonal dyadic block from Proposition \ref{prop:fixed-parity-qw-setup}, and assume
$$
M_1\le q^{\theta_1+o(1)},\qquad
M_2\le q^{\theta_2+o(1)},\qquad
N_1N_2\ll q^{1+\varepsilon_1},\qquad
M_1M_2\le q^{1-\eta_1}.
$$
Set $X=M_1N_2/N_1$. For a fixed factorisation $q=q_1q_2$ with $\mu^2(q_1)=1$, a fixed sign, $q_2>1$, and
$$
X\le q^{-\eta_2}
$$
with fixed $\eta_2>0$, the $k=0$ term contributes
$$
\ll_{\varepsilon_1} q^{\varepsilon_1}\frac{(M_1M_2)^{1/2}}{q}.
$$
The terms with $k\ne0$ contribute
$$
\ll_{\varepsilon_1}
q^{\varepsilon_1} q^{-3/2}M_2^{-1/2}X^{-1/2}
(Xq_2(1+X))^{3/4}A_{q_2}(2M_2)^{1/4}.
$$
If $q_2=1$, the sum over all $k\in\mathbb Z$ contributes
$$
\ll_{\varepsilon_1} q^{\varepsilon_1}\frac{(M_1M_2)^{1/2}}{q}.
$$
\end{prop}

\begin{proof}
Fix a factorisation $q=q_1q_2$ with $\mu^2(q_1)=1$ and a sign; Lemma \ref{lem:phistar-loss} absorbs the normalisation ratio $q/\varphi^*(q)$. Impose $(n_1,q_1)=1$ by $\mathbf 1_{(n_1,q_1)=1}=\sum_{r\mid(q_1,n_1)}\mu(r)$ and put $n_1=rn$. For each $r\mid q_1$, define
$$
F_r^{(j)}(k)=
\int_{\R}
W_1(u)V_j(N_1un_2/q)
\e\left(-\frac{kN_1u}{rq_2}\right)\,du
\ll_A q^{\varepsilon_1}
\left(1+\frac{|k|N_1}{rq_2}\right)^{-A}.
$$
Integration by parts permits restriction to $|k|\ll rq_2q^{\varepsilon_1}/N_1$, with error $O(q^{-100})$. The excluded term $m_1=m_2=n_1=1$ cannot occur when $X\le q^{-\eta_2}$, since it would give $X\asymp N_2\ge1$. Poisson summation gives
\begin{equation}\label{eq:poisson-block}
\begin{aligned}
\mathcal B_{\pm,r}(q_1,q_2)
&=
\frac{\mu^2(q_1)\varphi(q_2)N_1}{rq_2q^{3/2}(M_1M_2N_1N_2)^{1/2}}\\
&\quad\times
\sum_{\substack{m_1,m_2,n_2\\ (m_1m_2n_2,q)=1}}
\alpha_{m_1}\beta_{m_2}
W_2(n_2/N_2)
\sum_{k\in\Z}F_r^{(j)}(k)
S(\pm kn_2\overline{rm_1q_1},\bar m_2;q_2),
\end{aligned}
\end{equation}
where Lemma \ref{lem:corrected-complete-sum} evaluates the complete sum, and $\mathcal B_\pm=\sum_{r\mid q_1}\mu(r)\mathcal B_{\pm,r}$. For $k=0$, the Kloosterman sum is a Ramanujan sum at a unit, of absolute value at most $1$. Summing the remaining variables, using $N_1N_2\ll q^{1+\varepsilon_1}$ and $\varphi(q_2)/q_2\le1$, gives
$$
q^{\varepsilon_1}q^{-3/2}(M_1M_2N_1N_2)^{-1/2}N_1M_1M_2N_2\ll q^{\varepsilon_1}(M_1M_2)^{1/2}/q,
$$
after renaming the small exponent.

For $k\ne0$, group the triples by their first Kloosterman argument and set
$$
\nu_r(h)=
\frac1r\#\left\{(k,n_2,m_1):
\begin{array}{l}
0<|k|\ll rq_2q^{\varepsilon_1}/N_1,\quad n_2\asymp N_2,\ m_1\asymp M_1,\\
h\equiv kn_2\overline{rm_1q_1}\pmod{q_2}
\end{array}\right\}.
$$

Counting triples gives $\sum_h\nu_r(h)\ll Xq_2q^{\varepsilon_1}$. Two triples have the same $h$ exactly when $kn_2m_1'\equiv k'n_2'm_1\bmod q_2$. Put $u=kn_2m_1'$ and $v=k'n_2'm_1$; then $0<|u|,|v|\ll rXq_2q^{\varepsilon_1}$. There are $O(rXq_2q^{\varepsilon_1})$ choices of $u$ and at most $O(1+rXq^{\varepsilon_1})$ choices of $v\equiv u\bmod q_2$. Each has at most $q^{\varepsilon_1}$ factorisations after choosing the divisor-bound exponent sufficiently small. Including the factors $1/r$ and renaming the small exponent yields
$$
\begin{aligned}
\sum_h\nu_r(h)^2
&\ll q^{\varepsilon_1}\frac{rXq_2(1+rX)}{r^2}
\ll q^{\varepsilon_1}Xq_2(1+X),\\
\sum_h\nu_r(h)^{4/3}
&\le
\left(\sum_h\nu_r(h)\right)^{2/3}
\left(\sum_h\nu_r(h)^2\right)^{1/3}
\ll q^{\varepsilon_1} Xq_2(1+X)^{1/3}.
\end{aligned}
$$

For divisor-bounded $\gamma_m$ supported on $[M_2,2M_2]$, H\"older gives
$$
\sum_h\nu_r(h)
\left|
\sum_m\gamma_mS(\pm h,\overline{m};q_2)
\right|
\le
\left(\sum_h\nu_r(h)^{4/3}\right)^{3/4}
\left(\sum_h\left|\sum_m\gamma_mS(\pm h,\overline{m};q_2)\right|^4\right)^{1/4}.
$$
Expanding the fourth power gives
$$
\sum_{m^{(1)},\ldots,m^{(4)}}\gamma_{m^{(1)}}\overline{\gamma_{m^{(2)}}}\gamma_{m^{(3)}}\overline{\gamma_{m^{(4)}}}\sum_{h\bmod q_2}\prod_{i=1}^4S(\pm h,\overline{m^{(i)}};q_2),
$$
since the Kloosterman sums are real. Bound the coefficients by $q^{\varepsilon_1}$, enlarge each range to $1\le m^{(i)}\le2M_2$ while retaining $(m^{(i)},q_2)=1$, and replace $h$ by $-h$ for the negative sign. Thus the fourth power is at most $q^{\varepsilon_1}A_{q_2}(2M_2)$. The prefactor in \eqref{eq:poisson-block} is at most $r^{-1}q^{-3/2}M_2^{-1/2}X^{-1/2}$, and grouping by $h$ multiplies the sum by $r$. Together with the scalar in \eqref{eq:poisson-block}, this gives
$$
q^{\varepsilon_1}q^{-3/2}M_2^{-1/2}X^{-1/2}(Xq_2)^{3/4}(1+X)^{1/4}A_{q_2}(2M_2)^{1/4},
$$
which is stronger than the stated bound. Summing $r\mid q_1$ costs a divisor factor.

For $q_2=1$, the complete sum is $1$ and decay gives $\sum_k|F_r^{(j)}(k)|\ll rq^{\varepsilon_1}$. The Poisson factor $N_1/r$ then gives the same bound as $k=0$ above.
\end{proof}

\begin{lem}\label{lem:two-term-mixed-block-insertion}
Let $\omega_4,\omega_2\ge0$. Fix $\eta_1,\eta_2,\sigma>0$ and assume that the dyadic block lies in the complementary range
$$
\frac{N_2}{N_1}<q^\sigma\frac{M_1M_2}{q},\qquad
M_1M_2\le q^{1-\eta_1},\qquad
X=\frac{M_1N_2}{N_1}\le q^{-\eta_2}.
$$
Suppose that, for every $\varepsilon_1>0$,
$$
A_{q_2}(2M_2)\ll_{\varepsilon_1} q^{\varepsilon_1}
(q^{\omega_4}M_2^4q_2^{5/2}+q^{\omega_2}M_2^2q_2^3).
$$
If
$$
4\theta_1+6\theta_2+2\sigma+2\omega_4<3,
\qquad
2\theta_1+\theta_2+\sigma+\omega_2<1,
$$
\looseness=-1 then the fixed block $\mathcal B_\pm(q_1,q_2)$ is $O(q^{-\eta})$ for some $\eta>0$, uniformly in $q_2$, including $q_2=1$.
\end{lem}

\begin{proof}
The $k=0$ and $q_2=1$ terms are
$$
\ll q^{\varepsilon_1}\frac{(M_1M_2)^{1/2}}q\le q^{-1/2-\eta_1/2+\varepsilon_1}.
$$
 For $q_2>1$, insertion in Proposition \ref{prop:qw-preinsertion} gives for the quartic and quadratic terms, respectively,
$$
\begin{aligned}
&q^{\varepsilon_1}q^{-3/2+\omega_4/4}M_2^{1/2}X^{1/4}(1+X)^{3/4}q_2^{11/8},\\
&q^{\varepsilon_1}q^{-3/2+\omega_2/4}X^{1/4}(1+X)^{3/4}q_2^{3/2}.
\end{aligned}
$$
Since $q_2\le q$ and $X\le1$, these are at most
$$
q^{\varepsilon_1-1/8+\omega_4/4}M_2^{1/2}X^{1/4}\qquad\text{and}\qquad q^{\varepsilon_1+\omega_2/4}X^{1/4},
$$
and $X<q^{-1+\sigma}M_1^2M_2$ gives $X^{1/4}\le q^{-1/4+\sigma/4}M_1^{1/2}M_2^{1/4}$. The two contributions are therefore at most
$$
q^{\varepsilon_1-3/8+(\omega_4+\sigma)/4}M_1^{1/2}M_2^{3/4}
\qquad\text{and}\qquad
q^{\varepsilon_1-1/4+(\omega_2+\sigma)/4}M_1^{1/2}M_2^{1/4}.
$$
The resulting exponents at $M_i=q^{\theta_i+o(1)}$ are
$$
-\frac{3-(4\theta_1+6\theta_2+2\sigma+2\omega_4)}{8}+o(1),
\qquad
-\frac{1-(2\theta_1+\theta_2+\sigma+\omega_2)}{4}+o(1).
$$
Both are negative. Choose the small exponent and dyadic losses below half the smaller saving.
\end{proof}

The moment estimate must hold uniformly for every auxiliary modulus.

\begin{prop}\label{prop:qw-transfer}
Let $c_1,c_2\in\R$ and $0<\theta_1,\theta_2<1/2$ be fixed. Let $\omega_4,\omega_2\ge0$. Suppose that, for every $\varepsilon_1>0$, every auxiliary modulus $q_2$, and every dyadic length $M_2$,
$$
A_{q_2}(2M_2)
\ll
q^{\varepsilon_1}
\bigl(q^{\omega_4}(2M_2)^4q_2^{5/2}
+q^{\omega_2}(2M_2)^2q_2^3\bigr).
$$
Assume this estimate is uniform for the auxiliary factorisations and all dyadic boxes occurring in Proposition \ref{prop:fixed-parity-qw-setup}. Suppose that $\sigma,\eta_1,\eta_2>0$ are fixed and that every off-diagonal dyadic block satisfies
$$
M_1M_2\le q^{1-\eta_1},
$$
and that, in the range complementary to the summed large-ratio estimate,
$$
\frac{M_1N_2}{N_1}\le q^{-\eta_2}.
$$
Suppose further that
$$
4\theta_1+6\theta_2+2\sigma+2\omega_4<3,
\qquad
2\theta_1+\theta_2+\sigma+\omega_2<1.
$$
Then, for each $j\in\{0,1\}$,
$$
\mathcal M_1^{(j)}(q)=c_1+c_2+o(1)
$$
and
$$
\mathcal M_2^{(j)}(q)
=
(c_1+c_2)^2+\frac{c_1^2}{\theta_1}
+\frac{c_2^2}{\theta_2}+o(1).
$$
The implicit constants and $o(1)$-terms may depend on $c_1,c_2,\theta_1,\theta_2,\omega_4,\omega_2,\sigma,\eta_1,\eta_2,\varepsilon_1$ and the fixed smooth-weight seminorms, but not on $q,q_1,q_2$.
\end{prop}

\begin{proof}
Proposition \ref{prop:fixed-parity-qw-setup} gives the first moment, self-products and mixed diagonal. Lemma \ref{lem:fixed-factorisation-large-ratio} treats the large-ratio range summed over factorisations; Lemma \ref{lem:two-term-mixed-block-insertion} treats each complementary block, including $k=0$ and $q_2=1$. The strict exponent gaps absorb the $O((\log q)^C)$ dyadic boxes, $2^{\omega(q)}=q^{o(1)}$ factorisations and normalisation loss. The remaining error is $o(1)$, including $O(\log\log q/\log q)$.
\end{proof}

\begin{cor}\label{cor:qw-nv}
Let $0<\theta_1,\theta_2<1/2$. Suppose that, for $c_1=\theta_1$, $c_2=\theta_2$, and each $j\in\{0,1\}$, the mollifier in \eqref{eq:two-piece-mollifier} satisfies
$$
\mathcal M_1^{(j)}(q)=c_1+c_2+o(1)
$$
and
$$
\mathcal M_2^{(j)}(q)
=(c_1+c_2)^2+\frac{c_1^2}{\theta_1}
+\frac{c_2^2}{\theta_2}+o(1).
$$
Then
$$
\kappa(q)\ge
\frac{\theta_1+\theta_2}{1+\theta_1+\theta_2}+o(1).
$$
The same lower bound holds for each $\kappa_j(q)$ whose parity class is nonempty.
\end{cor}

\begin{proof}
Cauchy's inequality in each parity gives
$$
\frac{2}{\varphi^*(q)}
\#\{\chi\in\mathcal X^*(q;j):L(1/2,\chi)\ne0\}
\ge
\frac{|\mathcal M_1^{(j)}(q)|^2}{\mathcal M_2^{(j)}(q)}
=
\frac{(c_1+c_2)^2}{(c_1+c_2)^2+c_1^2/\theta_1+c_2^2/\theta_2}+o(1).
$$
Set $c_i=\theta_i$. The quotient is $(\theta_1+\theta_2)/(1+\theta_1+\theta_2)+o(1)$. Summing over the two parity classes gives $\kappa(q)$; Lemma \ref{lem:parity-count} replaces $2/\varphi^*(q)$ by $1/\#\mathcal X^*(q;j)$ with an $o(1)$ change.
\end{proof}

\section{Squarefree moduli}\label{sec:squarefree}

\subsection{Prime-local estimates and multiplicativity}

\begin{lem}\label{lem:kloo-mult}
If $(r,s)=1$ and $(b_1b_2b_3b_4,rs)=1$, then
$$
G_{rs}(b_1,b_2,b_3,b_4)=G_r(b_1,b_2,b_3,b_4)G_s(b_1,b_2,b_3,b_4),
$$
where
$$
G_q(b_1,b_2,b_3,b_4)=
\sum_{h\bmod q}\prod_{i=1}^4S(h,\overline{b_i};q).
$$
\end{lem}

\begin{proof}
Apply CRT multiplicativity of $S(a,b;c)$. Scaling removes the local unit factors, and the induced maps on $h\bmod r$ and $h\bmod s$ are permutations.
\end{proof}

\begin{lem}\label{lem:prime-local}
Let $p$ be prime and suppose $p\nmid b_1b_2b_3b_4$. If at least one of $b_1,b_2,b_3,b_4\bmod p$ occurs with multiplicity one, then
$$
|G_p(b_1,b_2,b_3,b_4)|\ll p^{5/2}.
$$
If no residue has multiplicity one, then
$$
|G_p(b_1,b_2,b_3,b_4)|\ll p^3.
$$
\end{lem}

\begin{proof}
For $p=2$ this is immediate. For odd $p$, apply Proposition \ref{prop:fgkm-prime-local} with $a_i=\overline{b_i}$.
\end{proof}

For squarefree $q$, define
$$
D_q(\mathbf b)=
\prod_{\substack{p\mid q\\
\{b_1,b_2,b_3,b_4\}\bmod p\text{ has no singleton}}}p.
$$

\begin{lem}\label{lem:sqfree-pointwise}
Let $q$ be squarefree and $(b_1b_2b_3b_4,q)=1$. Then
$$
|G_q(b_1,b_2,b_3,b_4)|
\ll q^{5/2+o(1)}D_q(\mathbf b)^{1/2}.
$$
\end{lem}

\begin{proof}
Multiply the local bounds of Lemma \ref{lem:prime-local} using Lemma \ref{lem:kloo-mult}.
\end{proof}

\subsection{Counting the pairing congruences}

\begin{prop}\label{prop:D-count}
Let $q$ be squarefree. For every $B\ge1$,
$$
\sum_{1\le b_1,b_2,b_3,b_4\le B}D_q(b_1,b_2,b_3,b_4)^{1/2}
\ll q^{o(1)}(B^4+B^2q^{1/2}).
$$
\end{prop}

\begin{proof}
Majorise $D_q(\mathbf b)^{1/2}$ by $\sum_{d\mid q,\ d\mid D_q(\mathbf b)}d^{1/2}$. At each $p\mid d$, the absence of a singleton is covered by the pairings $(12)(34)$, $(13)(24)$ and $(14)(23)$. Assign one to each prime and write $d=d_{12}d_{13}d_{14}$, with pairwise coprime factors. For $u=b_2-b_1$, $v=b_3-b_1$, $w=b_4-b_1$, the conditions are
$$
\begin{gathered}
d_{12}\mid u,\quad d_{13}\mid v,\quad d_{14}\mid w,\\
d_{12}\mid v-w,\quad d_{13}\mid u-w,\quad d_{14}\mid u-v.
\end{gathered}
$$
Choose $b_1$, then $u$ in the class $0$ modulo $d_{12}$, $v$ in one class modulo $d_{13}d_{14}$, and $w$ in one class modulo $d$, by the Chinese remainder theorem. This gives
$$
\ll B\left(\frac{B}{d_{12}}+1\right)
\left(\frac{B}{d_{13}d_{14}}+1\right)
\left(\frac{B}{d}+1\right)
\ll\frac{B^4}{d^2}+B^2.
$$
For the last inequality, the cubic terms have denominators $d,d_{12}d,d_{13}d_{14}d$; compare respectively $d,d_{13}d_{14},d_{12}$ with $B$ to bound each by $B^4/d^2+B^2$. Lower-degree terms are $O(B^2)$. Summing over assignments and divisors gives
$$
B^4\sum_{d\mid q}\frac{3^{\omega(d)}}{d^{3/2}}
+
B^2\sum_{d\mid q}3^{\omega(d)}d^{1/2}
\ll q^{o(1)}(B^4+B^2q^{1/2}).
$$
\end{proof}

\subsection{Completion of the squarefree argument}

\begin{proof}[Proof of Theorem~\ref{thm:sqfree-average}]
By Lemma \ref{lem:sqfree-pointwise} and Proposition \ref{prop:D-count},
$$
A_q(B)\ll q^{5/2+o(1)}(B^4+B^2q^{1/2})
=q^{o(1)}(B^4q^{5/2}+B^2q^3).
$$
Replacing $q^{o(1)}$ by $q^\ep$ with the constant depending on $\ep$ gives the stated estimate.
\end{proof}

\begin{prop}\label{prop:sqfree-moment}
Let $q$ be squarefree and $q\not\equiv2\bmod4$. Let $\sigma>0$. If
$$
0<\theta_1,\theta_2<1/2,
\qquad
4\theta_1+6\theta_2+2\sigma<3,
\qquad
2\theta_1+\theta_2+\sigma<1,
$$
then the moment formulae of Proposition \ref{prop:qw-transfer} hold for each parity.
\end{prop}

\begin{proof}
Every auxiliary $q_2$ is squarefree, so Theorem \ref{thm:sqfree-average} gives $\omega_4=\omega_2=0$ in Proposition \ref{prop:qw-transfer}. Writing $\Delta=1-(2\theta_1+\theta_2+\sigma)>0$, small dyadic losses give $M_1M_2\le q^{1-\Delta/2}$ and, outside the large-ratio range, $X<q^{-1+\sigma}M_1^2M_2\le q^{-\Delta/2}$. All transfer hypotheses hold.
\end{proof}

\begin{proof}[Proof of Theorem~\ref{thm:sqfree-nonvanishing}]
Take $\theta_1=3/8-\rho$, $\theta_2=1/4-\rho$ and $0<\sigma<3\rho$. Then ${4\theta_1+6\theta_2=3-10\rho}$ and $2\theta_1+\theta_2=1-3\rho$. Proposition \ref{prop:sqfree-moment} and Corollary \ref{cor:qw-nv} give
$$
\kappa(q)\ge\frac{5/8-2\rho}{13/8-2\rho}+o(1).
$$
Take $\rho$ sufficiently small.
\end{proof}

\section{Prime-square Kloosterman estimates}\label{sec:p2-kloosterman}

\subsection{Stationary phase and Hensel roots}

Fix an odd prime $p$ and a nonsquare class $\nu\bmod p$, represented modulo $p^2$. A unit modulo $p^2$ is a square exactly when its reduction modulo $p$ is a square, so $C_p=\{1,\nu\}$ represents the two square classes of $(\Z/p^2\Z)^\times$. For $c\in C_p$ and $1\le B<p$, put
$$
\begin{aligned}
A_c(B)&=\{1\le b\le B:cb^{-1}\hbox{ is a square modulo }p\},\\
S_c(B)&=\{x\bmod p^2:x^2\equiv c\bar{b}\bmod {p^2}\hbox{ for some }b\in A_c(B)\},\\
\mathcal W_c(B)&=
\sum_{x_1,x_2,x_3,x_4\in S_c(B)}
|c_{p^2}(x_1+x_2+x_3+x_4)|,
\end{aligned}
$$
where $c_{p^2}(n)$ is the Ramanujan sum.

\begin{lem}\label{lem:p2-kloosterman-eval}
Let $p$ be an odd prime and $(ab,p)=1$. Then
$$
S(a,b;p^2)
=
p\sum_{\substack{x\bmod p^2\\ x^2\equiv ab\bmod {p^2}}}
\e\left(\frac{2x}{p^2}\right).
$$
\looseness=-1 The sum is empty if $ab$ is a quadratic nonresidue modulo $p$, and otherwise has two terms.
\end{lem}

\begin{proof}
Take $\nu=2$ in \eqref{eq:reciprocal-odd-branch-sum}, where $\tau_p(A;p^\nu)=1$; compare \cite[Lemma~22]{BM2015}.
\end{proof}

\begin{lem}\label{lem:p2-local-ramanujan}
Let $p$ be an odd prime and $p\nmid b$. Then $S(h,b;p^2)=0$ whenever $p\mid h$. If $p\nmid h$, Lemma \ref{lem:p2-kloosterman-eval} gives
$$
S(h,b;p^2)
=
p\sum_{\substack{x\bmod p^2\\ x^2\equiv hb\bmod {p^2}}}
\e\left(\frac{2x}{p^2}\right),
$$
with two roots when $hb$ is a square modulo $p$ and no roots otherwise. Also
$$
c_{p^2}(n)=
\begin{cases}
p(p-1),&p^2\mid n,\\
-p,&p\mid n,\ p^2\nmid n,\\
0,&p\nmid n.
\end{cases}
$$
\end{lem}

\begin{proof}
Use Lemma \ref{lem:p2-kloosterman-eval} and Proposition \ref{prop:weil-estermann}. Grouping units modulo $p^2$ by their reductions modulo $p$ gives the Ramanujan sum.
\end{proof}

\begin{lem}\label{lem:p2-root-reduction}
Let $\mathbf b=(b_1,b_2,b_3,b_4)$ with $p\nmid b_1b_2b_3b_4$. Then one has the identity
$$
G_{p^2}(\mathbf b)
=
\frac{p^4}{2}\sum_{c\in C_p}
\sum_{\substack{z_1,\ldots,z_4\bmod p^2\\
z_i^2\equiv c\overline{b_i}\bmod {p^2}\ (1\le i\le4)}}
c_{p^2}\left(2\sum_{i=1}^4z_i\right).
$$
Consequently, for $1\le B<p$,
$$
\begin{aligned}
A_{p^2}(B)
&=\frac{p^4}{2}\sum_{c\in C_p}
\sum_{(b_1,b_2,b_3,b_4)\in A_c(B)^4}
\left|
\sum_{\substack{z_1,\ldots,z_4\bmod p^2\\
z_i^2\equiv c\overline{b_i}\bmod {p^2}\ (1\le i\le4)}}
c_{p^2}(z_1+z_2+z_3+z_4)
\right|\\
&\le \frac{p^4}{2}\sum_{c\in C_p}\mathcal W_c(B).
\end{aligned}
$$
\end{lem}

\begin{proof}
Only units $h$ contribute, by Lemma~\ref{lem:p2-local-ramanujan}. On each square class, $h=cu^2$ is a two-to-one parametrisation, and the roots in Lemma \ref{lem:p2-kloosterman-eval} are $uz_i$. Summing the product over $u$ gives $c_{p^2}(2\sum_iz_i)$ with prefactor $p^4/2$, proving the identity. Since $2$ is a unit, $c_{p^2}(2n)=c_{p^2}(n)$. A quadruple belongs to at most one $A_c(B)^4$; otherwise some local Kloosterman factor vanishes. This gives the stated equality for $A_{p^2}(B)$. For fixed $c$, distinct $b\le B<p$ have disjoint root sets, because $z\bmod p$ determines $b$. Their union is $S_c(B)$, so the triangle inequality proves the final bound.
\end{proof}

\begin{lem}\label{lem:root-reduction-bijection}
For $1\le B<p$ and $c\in C_p$, reduction modulo $p$ is injective on $S_c(B)$. Its image is the set of $y\in\F_p^\times$ for which $cy^{-2}\bmod p$ has a representative in $[1,B]$.
\end{lem}

\begin{proof}
If $x^2\equiv c\bar{b}\bmod {p^2}$ and $y=x\bmod p$, then $b\equiv cy^{-2}\bmod p$. Since $1\le b\le B<p$, the residue determines $b$ uniquely. Conversely, every residue $y$ described in the lemma gives this integer $b$. The derivative of $T^2-c\bar{b}$ at $T=y$ is $2y$, a unit modulo $p$, so Hensel's lemma gives a unique lift modulo $p^2$.
\end{proof}

Let $\mathcal W_c^{\rm diag}(B)$ be the part of $\mathcal W_c(B)$ in which some pair satisfies $x_i+x_j\equiv0\bmod {p^2}$, and let $\mathcal W_c^\sharp(B)$ be the complementary part.

\subsection{The exceptional algebraic locus}

\begin{lem}\label{lem:diag}
For $1\le B<p$,
$$
\sum_{c\in C_p}\mathcal W_c^{\rm diag}(B)\ll B^2p^2.
$$
\end{lem}

\begin{proof}
A non-zero Ramanujan weight forces $p\mid\sum_ix_i$ and is at most $p^2$. Fix, for example, $x_1+x_2\equiv 0\bmod {p^2}$. There are $O(B)$ choices for this pair, since $x_1\in S_c(B)$ determines $x_2$, and then $x_3+x_4\equiv 0\bmod p$. Lemma \ref{lem:root-reduction-bijection} gives $O(B)$ choices for the latter pair. Sum over the six pairs and two square classes.
\end{proof}

For $a_1,a_2,a_3,a_4$, define
\begin{multline*}
R(a_1,a_2,a_3,a_4)=\Bigl(\bigl(a_3a_4(a_1+a_2)-a_1a_2(a_3+a_4)\bigr)^2-4a_1a_2a_3a_4(a_1a_2+a_3a_4)\Bigr)^2\\
-64(a_1a_2a_3a_4)^3.
\end{multline*}

Define
$$
N_c(B;p)=
\#\{(a_1,a_2,a_3,a_4)\in A_c(B)^4:
p\mid R(a_1,a_2,a_3,a_4)\}.
$$

\begin{lem}\label{lem:eliminant-identity}
Let $A_i^2=a_i$. Then
$$
R(a_1,a_2,a_3,a_4)
=
\prod_{s_2,s_3,s_4\in\{\pm1\}}
(A_1A_2A_3+s_2A_1A_2A_4
+s_3A_1A_3A_4+s_4A_2A_3A_4).
$$
Consequently $R$ is a non-zero homogeneous integer polynomial of degree $12$. For positive integers $a_i$, $R(a_1,a_2,a_3,a_4)=0$ if and only if there are signs $s_i$ such that
$$
\frac{s_1}{\sqrt{a_1}}+
\frac{s_2}{\sqrt{a_2}}+
\frac{s_3}{\sqrt{a_3}}+
\frac{s_4}{\sqrt{a_4}}=0.
$$
Moreover, over any field of odd characteristic and for non-zero $a_i$, the existence of such a signed reciprocal relation implies $R(a_1,a_2,a_3,a_4)=0$ in that field.
\end{lem}

\begin{proof}
Put $U=A_1A_2A_3$, $V=A_1A_2A_4$, $W=A_1A_3A_4$, $Z=A_2A_3A_4$ and ${A=U^2+V^2-W^2-Z^2}$. Multiplying first over $s_3,s_4$ gives $(A+2s_2UV)^2-4W^2Z^2$. The product over $s_2$ is therefore
$$
\begin{aligned}
&(A^2+4U^2V^2-4W^2Z^2)^2-16A^2U^2V^2\\
&\qquad=(A^2-4U^2V^2-4W^2Z^2)^2-64U^2V^2W^2Z^2.
\end{aligned}
$$
Substitute
$$
\begin{gathered}A=a_1a_2(a_3+a_4)-a_3a_4(a_1+a_2),\\ U^2V^2=(a_1a_2a_3a_4)a_1a_2,\qquad W^2Z^2=(a_1a_2a_3a_4)a_3a_4,\qquad U^2V^2W^2Z^2=(a_1a_2a_3a_4)^3.\end{gathered}
$$
This is $R$, an integer polynomial of degree $12$, non-zero since $R(1,1,1,4)=945$. Dividing a vanishing factor by $A_1A_2A_3A_4$ gives the reciprocal relation, and conversely. The identity over an algebraic closure also proves the odd-characteristic implication for non-zero $a_i$.
\end{proof}

\begin{lem}\label{lem:surface}
For $1\le B<p$ and $c\in C_p$,
$$
\mathcal W_c^\sharp(B)\le 16p^2N_c(B;p).
$$
The constant is uniform in $p$, $B$, $c$, and the chosen nonsquare representative.
\end{lem}

\begin{proof}
A non-zero weight forces $\sum_iy_i=0$ in $\F_p$, where $y_i=x_i\bmod p$. Put $a_i\equiv cy_i^{-2}\bmod p$, $1\le a_i\le B$, and choose $u^2=c$ in an algebraic closure. Writing ${y_i=s_iu/\sqrt{a_i}}$ for signs $s_i$ and dividing by $u$, Lemma \ref{lem:eliminant-identity} gives $p\mid R(\mathbf a)$. Here $a_i\ne0$ since $B<p$. Each tuple has at most $2^4$ reductions, each lifting uniquely by Lemma \ref{lem:root-reduction-bijection}, and each weight is at most $p^2$.
\end{proof}

\begin{lem}\label{lem:squarefree-radical-independence}
Let $r_1,\ldots,r_t$ be distinct squarefree positive integers. Then $\sqrt{r_1},\ldots,\sqrt{r_t}$ are linearly independent over $\Q$.
\end{lem}

\begin{proof}
Let $\ell_1,\ldots,\ell_n$ be the primes dividing the $r_i$. Inductively, $K_j=\Q(\sqrt{\ell_1},\ldots,\sqrt{\ell_j})$ has the monomial basis $\prod_{\ell\in E}\sqrt\ell$, $E\subseteq\{\ell_1,\ldots,\ell_j\}$, and all independent sign changes as automorphisms. If $\sqrt{\ell_j}\in K_{j-1}$, its rational square makes it an eigenvector for every sign change. These act on the basis monomials by distinct characters, so it is a rational multiple of one monomial. Squaring contradicts the odd $\ell_j$-adic valuation. Thus adjoining $\sqrt{\ell_j}$ doubles the degree, gives the stated basis, and extends each sign change in both ways. The $\sqrt{r_i}$ are distinct basis monomials.
\end{proof}

\begin{lem}\label{lem:integer-zero-count}
For every $\ep>0$,
$$
\#\{(a_1,a_2,a_3,a_4)\in[1,B]^4:
R(a_1,a_2,a_3,a_4)=0\}
\ll_\ep B^{2+\ep}.
$$
\end{lem}

\begin{proof}
By Lemma \ref{lem:eliminant-identity}, an integer zero satisfies $\sum_is_i/\sqrt{a_i}=0$ for some signs. Write $a_i=r_ix_i^2$ with $r_i$ squarefree. Lemma \ref{lem:squarefree-radical-independence} gives $\sum_{i:r_i=r}s_i/x_i=0$ for each $r$. Thus the indices form two pairs with equal squarefree part, or all four parts coincide. In the first case each pair has opposite signs and equal $x_i$, hence equal $a_i$; the three pairings give $O(B^2)$ tuples. In the second case put $a_i=rx_i^2$. For fixed signs and $x_1,x_2,x_3\le(B/r)^{1/2}$ the relation determines at most one positive $x_4$. Summing gives $O(B^{3/2}\sum_{r\ge1}r^{-3/2})=O(B^{3/2})$. The finitely many sign patterns therefore give $O(B^2)$ in total.
\end{proof}

\subsection{Averaging over primes}

\begin{prop}\label{prop:prime-average}
For every $\ep>0$, uniformly for $P\ge2$ and $1\le B\le P^{1/2}$,
$$
\sum_{P<p\le2P}\sum_{c\in C_p}N_c(B;p)
\ll_\ep \pi(P)B^{2+\ep}+B^4.
$$
\end{prop}

\begin{proof}
Integer zeros contribute $\ll_\ep \pi(P)B^{2+\ep}$ by Lemma \ref{lem:integer-zero-count}. Otherwise
$$
0<|R(\mathbf a)|\ll B^{12}\le P^6,
$$
and hence has only $O(1)$ prime divisors in $(P,2P]$. Summing over $B^4$ tuples and two square classes gives the remaining $O(B^4)$ term.
\end{proof}

\begin{prop}\label{prop:almost-all-N}
Fix $\lambda,\delta>0$ and $0<\xi<\min(\lambda,1)$. There is a set of primes $\mathcal E$ such that
$$
\#\{p\le X:p\in\mathcal E\}
\ll_{\lambda,\delta,\xi}\frac{X^{1-\xi}}{\log X},
$$
and, for every prime $p\notin\mathcal E$, uniformly for all $1\le B\le p^{1/2-\delta}$,
$$
\sum_{c\in C_p}N_c(B;p)
\ll_{\lambda,\delta,\xi} p^\lambda\left(B^2+\frac{B^4}{p}\right).
$$
\end{prop}

\begin{proof}
Choose $\xi<\xi_0<\min(\lambda,1)$. On $(P,2P]$, for a dyadic endpoint $H\le (2P)^{1/2-\delta/2}$, Proposition \ref{prop:prime-average} and Markov's inequality bound the number of failures of $\sum_cN_c(H;p)\le P^\lambda(H^2+H^4/P)$ by
$$
\ll_\ep P^{-\lambda}
\frac{\pi(P)H^{2+\ep}+H^4}{H^2+H^4/P}
\ll_{\lambda,\xi_0} \frac{P^{1-\xi_0}}{\log P},
$$
with $\ep<\lambda-\xi_0$. Indeed the quotient is at most $\pi(P)H^\ep+P\ll P^{1+\ep/2}/\log P$, since $H\le P^{1/2}$, and $\lambda-\ep/2>\xi_0$. For sufficiently large $P$ the endpoint range is within $H\le P^{1/2}$, as required by that proposition. Take the union over $O_\delta(\log P)$ endpoints; absorbing this factor into $P^{\xi_0-\xi}$ gives $O(P^{1-\xi}/\log P)$ exceptional primes. Sum over dyadic prime intervals, using $\sum_{2^j\le X}2^{j(1-\xi)}/j\ll X^{1-\xi}/\log X$; the finitely many small intervals are absorbed into the constant.

For arbitrary $B\le p^{1/2-\delta}$, choose a dyadic $B\le H<2B$. For large $p$ it lies in the endpoint range; monotonicity and $H^2+H^4/P\ll B^2+B^4/p$ give the bound simultaneously for all $B$.
\end{proof}

\begin{proof}[Proof of Theorem~\ref{thm:p2-average}]
Choose $\mathcal E$ from Proposition \ref{prop:almost-all-N} with $\lambda=\ep$ and the given $\delta,\xi$. Lemmas \ref{lem:p2-root-reduction}, \ref{lem:diag} and \ref{lem:surface} give, simultaneously for all permitted $B$,
$$
A_{p^2}(B)\ll p^4\left(B^2p^2+p^2\sum_cN_c(B;p)\right)\ll p^{\ep}(B^4p^5+B^2p^6),
$$
where the last step uses Proposition~\ref{prop:almost-all-N}. Here $B<p$, as Lemma~\ref{lem:p2-root-reduction} requires, since $B\le p^{1/2-\delta}$.
\end{proof}

\begin{lem}\label{lem:p2-range-dependencies}
Let $\rho>0$. If $q=p^2$, $\theta_2\le1/4-\rho$, and $B=2M_2\ll q^{\theta_2+o(1)}$, then, for every fixed $0<\delta_0<2\rho$,
$$
B\le p^{1/2-\delta_0}
$$
for all sufficiently large $p$. If $q=mp^2$ with $m$ fixed, the same conclusion holds with constants depending on $m$.

If $q=rp^2$, $r\le p^\alpha$, and $\theta_2\le1/4-\rho$, then
$$
B\ll p^{(2+\alpha)(1/4-\rho)+o(1)}.
$$
Thus $B\le p^{1/2-\delta_0}$ whenever
$$
(2+\alpha)\left(\frac{1}{4}-\rho\right)<\frac{1}{2}-\delta_0.
$$
\end{lem}

\begin{proof}
Substitute $q=p^2$, $mp^2$, or $rp^2\le p^{2+\alpha}$ into $B\ll q^{1/4-\rho+o(1)}$; fixed constants and the $o(1)$ loss are absorbed by the strict exponent gaps.
\end{proof}

\section{Prime-square conductors}\label{sec:p2-nonvanishing}

\begin{proof}[Proof of Theorem~\ref{thm:p2-nonvanishing}]
With $\rho>0$ small, take
$$
\theta_1=3/8-\rho,\quad \theta_2=1/4-\rho,\quad \lambda=\sigma=\rho,\quad \delta=\rho/2,\quad \xi=\min(\rho/2,1/2).
$$
 Lemma \ref{lem:p2-range-dependencies} gives $2M_2\le p^{1/2-\delta}$. Choose the exceptional set from Theorem~\ref{thm:p2-average} with $\ep=\lambda$ and these $\delta,\xi$, and put $\xi_1=\xi$. The only auxiliary modulus is $q_2=p^2$, for which the transfer estimate has $\omega_4=\omega_2=\lambda/2$. After dyadic losses, $M_1M_2\le q^{5/8-3\rho/2}$ and in the complementary range $X<q^{-1+\sigma}M_1^2M_2\le q^{-3\rho/2}$. The two inequalities are
$$
4\theta_1+6\theta_2+2\sigma+\lambda=3-7\rho<3,
\qquad
2\theta_1+\theta_2+\sigma+\lambda/2=1-3\rho/2<1.
$$
Proposition \ref{prop:qw-transfer} and Corollary \ref{cor:qw-nv} give $\kappa(p^2)\ge(5/8-2\rho)/(13/8-2\rho)+o(1)$. Choose $\rho$ sufficiently small in terms of $\ep$.
\end{proof}

\section{Fixed multiples and moduli with a small powerful part}
\label{sec:small-nsf}

\subsection{Fixed multiples of prime squares}

\begin{lem}\label{lem:crt-stability}
\looseness=-1 Fix $m\ge1$ and let $p\nmid m$. There is a constant $C_m$ such that for every $B<p$,
$$
\sum_{\substack{1\le b_1,b_2,b_3,b_4\le B\\ (b_1b_2b_3b_4,mp^2)=1}}
|G_{mp^2}(b_1,b_2,b_3,b_4)|
\le C_m A_{p^2}(B).
$$
\end{lem}

\begin{proof}
Lemma \ref{lem:kloo-mult} gives $G_{mp^2}(\mathbf b)=G_m(\mathbf b)G_{p^2}(\mathbf b)$. The bound follows with
$$
C_m=m\max_{\substack{a,b\bmod m\\ (b,m)=1}}|S(a,b;m)|^4,
$$
after dropping $(b_i,m)=1$ following absolute values.
\end{proof}

\begin{proof}[Proof of Theorem~\ref{thm:fixed-m-main}]
Use the parameters, the exceptional set $\mathcal E$ and the exponent $\xi_1$ of the proof of Theorem \ref{thm:p2-nonvanishing}. Let $p\notin\mathcal E$ with $p>m$, so that $p\nmid m$. Every auxiliary modulus is $q_2=dp^2$, with $d\mid m$ and $(d,m/d)=1$. Lemma \ref{lem:crt-stability} transfers Theorem~\ref{thm:p2-average}, applied with $\ep=\lambda$, to these finitely many moduli. Since $p^2\le q_2$ and $p^\lambda\le q^{\lambda/2}$, this gives
$$
A_{q_2}(2M_2)\ll_m q^{\lambda/2}
\bigl((2M_2)^4q_2^{5/2}+(2M_2)^2q_2^3\bigr).
$$
Since $m$ is fixed, Lemma \ref{lem:p2-range-dependencies} gives the same length range. The bounds for $M_1M_2$ and $X$, and the two exponent inequalities, are unchanged. Proposition \ref{prop:qw-transfer} and Corollary \ref{cor:qw-nv} conclude.
\end{proof}

\subsection{A sharper bound for the powerful part}
\label{sec:noloss}

We now count root configurations to retain a cubic term for short lengths, using the local prime-power formulae \eqref{eq:reciprocal-gauss-sign}--\eqref{eq:reciprocal-two-branch}.

\begin{lem}\label{lem:global-root}
Let $s$ be odd and powerful and $(am,s)=1$. Write $s=\prod_jq_j$ with $q_j=p_j^{\nu_j}$, let $R_j=s/q_j$ and let $t_j\equiv\overline{R_j}\bmod{q_j}$. Then
$$
S(a,m;s)=s^{1/2}\sum_{\substack{y\bmod s\\ y^2\equiv am}}\chi_s(y)\,\e\left(\frac{2y}{s}\right),
\qquad
\chi_s(y)=\prod_j\tau_{p_j}(t_jy;q_j),
$$
with $\tau_p(A;p^\nu)$ as in \eqref{eq:reciprocal-gauss-sign}. Here $|\chi_s(y)|=1$; the sum is empty unless $am$ is a square modulo $s$.
\end{lem}

\begin{proof}
\looseness=-1 Twisted multiplicativity gives $S(a,m;s)=\prod_jS(a\overline{R_j},m\overline{R_j};q_j)$, whose two arguments have product $am\,t_j^2$. Recombining the two branches of \eqref{eq:reciprocal-odd-branch-sum} expresses each local factor as
$$
q_j^{1/2}\sum_{x_j^2\equiv am t_j^2}\tau_{p_j}(x_j;q_j)\e\Bigl(\frac{2x_j}{q_j}\Bigr);
$$
substituting $x_j=t_jy_j$ and using the Chinese remainder theorem for the roots together with $\sum_j\overline{R_j}/q_j\equiv1/s\bmod1$ gives the statement.
\end{proof}

\begin{lem}\label{lem:chi-cancel}
\looseness=-1 For $s$ odd and powerful and $u$ a unit modulo $s$, $\chi_s(uy)=\psi_s(u)\chi_s(y)$, where
$$
\psi_s(u)=\prod_{\substack{p^{\nu}\Vert s\\ 2\nmid\nu}}\left(\frac up\right)
$$
is a real character. In particular $\psi_s(u)^4=1$.
\end{lem}

\begin{proof}
At a prime with $\nu$ even $\tau_p(A;p^\nu)=1$; at a prime with $\nu$ odd, $\tau_p(A;p^\nu)$ is $\left(\frac Ap\right)$ times a constant depending only on $p$, and the Legendre symbol is completely multiplicative.
\end{proof}

\begin{prop}\label{prop:general-root-reduction}
Let $s$ be odd and powerful and $(b_1b_2b_3b_4,s)=1$, and let $C_s$ be a set of representatives for the square classes of $(\Z/s\Z)^\times$. Then
$$
G_s(\mathbf b)=\frac{s^2}{2^{\omega(s)}}\sum_{c\in C_s}
\sum_{\substack{z_1,\ldots,z_4\bmod s\\ z_i^2\equiv c\overline{b_i}}}
\left(\prod_{i=1}^4\chi_s(z_i)\right)c_s\left(2\sum_{i=1}^4z_i\right),
$$
and at most one $c$ contributes. For $s=p^2$ this is Lemma \ref{lem:p2-root-reduction}.
\end{prop}

\begin{proof}
Only units $h$ contribute, by Proposition \ref{prop:weil-estermann}. Write $h=cu^2$, a $2^{\omega(s)}$-to-one parametrisation of the square class of $c$, and $y_i=uz_i$ in Lemma \ref{lem:global-root}. By Lemma \ref{lem:chi-cancel} the four Gauss factors contribute $\psi_s(u)^4=1$, so $u$ appears only in the additive character and the sum over $u$ is the Ramanujan sum $c_s(2\sum z_i)$.
\end{proof}

\begin{prop}\label{prop:general-fourth}
For $s$ odd and powerful and every $B\ge1$,
$$
\sum_{\substack{\mathbf b\in[1,B]^4\\(b_1b_2b_3b_4,s)=1}}|G_s(\mathbf b)|
\ \ll\ s^{o(1)}\left(s^2B^4+s^3B^3\right)
\ \le\ s^{5/2+o(1)}B^4
\qquad\text{whenever } B\ge s^{1/2}.
$$
\end{prop}

\begin{proof}
\looseness=-1 The case $s=1$ is immediate. For odd $s>1$, $|c_s(2\sum_iz_i)|\le\sum_{D\mid s}D\,1_{D\mid\sum_iz_i}$. Fix $D\mid s$ and a square class. The first three integers and roots have at most $s^{o(1)}B^3$ choices. Their sum determines $z_4\bmod D$; if this is a unit it determines $b_4\bmod D$, giving at most $B/D+1$ integers, each with $s^{o(1)}$ roots. Otherwise there is no solution. Proposition \ref{prop:general-root-reduction} therefore gives
$$
\sum_{\mathbf b}|G_s(\mathbf b)|
\ll s^{2+o(1)}\sum_{D\mid s}(B^4+DB^3)
\ll s^{o(1)}(s^2B^4+s^3B^3).
$$
For $B\ge s^{1/2}$ both terms are at most $s^{5/2}B^4$.
\end{proof}

\begin{prop}\label{prop:nsf-noloss}
Let $q=sr$ with $(s,r)=1$, $r$ squarefree and $s$ powerful. Then
$$
A_q(B)\ll_\ep q^\ep
\left(B^4q^{5/2}+B^3q^{5/2}s^{1/2}+B^2q^3\right)
$$
for every $B\ge1$ and every $\ep>0$. In particular, if $B\ge s^{1/2}$, then
$$
A_q(B)\ll_\ep q^\ep\left(B^4q^{5/2}+B^2q^3\right),
$$
with no factor $s^{1/2}$.
\end{prop}

\begin{proof}
First suppose $s$ is odd. By multiplicativity and Lemma \ref{lem:sqfree-pointwise},
$$
|G_q(\mathbf b)|\ll r^{5/2+o(1)}|G_s(\mathbf b)|
\sum_{\substack{d\mid r\\d\mid D_r(\mathbf b)}}d^{1/2}.
$$
For each $d=d_{12}d_{13}d_{14}$, retain the pairing assignment of Proposition \ref{prop:D-count}; retain also each Ramanujan level $D\mid s$. Fix a square class. The first three integers and their roots have at most $s^{o(1)}B(B/d_{12}+1)(B/(d_{13}d_{14})+1)$ choices. The root condition fixes $b_4\bmod D$, and the remaining pairing conditions fix $b_4\bmod d$. Since $(d,D)=1$, the fourth integer has at most $B/(dD)+1$ choices, each with $s^{o(1)}$ roots. Thus the count is
$$
s^{o(1)}B\left(\frac B{d_{12}}+1\right)
\left(\frac B{d_{13}d_{14}}+1\right)
\left(\frac B{dD}+1\right)\ll s^{o(1)}
\left(\frac{B^4}{d^2D}+\frac{B^3}{d}+B^2\right).
$$
With cubic coefficients at most $1/d$ and lower-degree terms at most $B^2$, Proposition \ref{prop:general-root-reduction} gives
$$
\begin{aligned}
A_q(B)&\ll q^{o(1)}s^2r^{5/2}
\sum_{D\mid s}D\sum_{d\mid r}3^{\omega(d)}d^{1/2}
\left(\frac{B^4}{d^2D}+\frac{B^3}{d}+B^2\right)\\
&\ll q^{o(1)}(s^2r^{5/2}B^4+s^3r^{5/2}B^3+q^3B^2).
\end{aligned}
$$
Here $\sum_{D\mid s}1\ll s^{o(1)}$, $\sum_{D\mid s}D\ll s^{1+o(1)}$, and the three $d$-sums, with exponents $-3/2,-1/2,1/2$, are respectively $O(1)$, $r^{o(1)}$ and $r^{1/2+o(1)}$. This proves the claim for odd $s$, including $s=1$ by retaining only $D=1$.

If $2^\nu\Vert s$ with $\nu\ge8$, use \eqref{eq:reciprocal-two-gauss-sign}--\eqref{eq:reciprocal-two-branch}. Only odd $h$ contribute. In a square class supporting all four factors, $h=cu^2$ with $u\equiv1\bmod4$ is two-to-one. Write $z_i=(c\overline{b}_i)_{1/2}\bmod2^{\nu-1}$, congruent to $1\bmod4$; the corresponding root for $h\overline{b}_i$ is $uz_i$. For even $\nu$ the four Gauss factors have product $\e(\sum_i\epsilon_i/8)$, independent of $u$. For odd $\nu$ their product combines with the phases as
$$
\e\left(\frac{(2^{\nu-4}+1)u\sum_i\epsilon_i z_i}{2^{\nu-1}}\right),
$$
since $\tau_2(\epsilon_iuz_i;2^\nu)=\e(\epsilon_iuz_i/8)$ and the phase of the branch is $\e(\epsilon_iuz_i/2^{\nu-1})$. The coefficient is odd (and is $1$ for even $\nu$, apart from the constant). Putting $u=1+4v$ with $v\bmod2^{\nu-2}$, the $u$-sum vanishes unless $2^{\nu-3}\mid\sum_i\epsilon_iz_i$, and otherwise has size $2^{\nu-2}$. Thus, since each branch has modulus $2^{(\nu+1)/2}$, each square class and sign pattern has weight at most $(\sqrt{2^{\nu+1}})^4\,2^{\nu-2}/2=4(2^\nu)^22^{\nu-3}$.

Combine this condition with each odd level by CRT. The resulting level $D\mid s$ has weight at most $4s^2D$, and the first three roots again fix $b_4\bmod D$. The preceding pairing count applies, with only four square classes and sixteen sign patterns at two; the level sums have the same bounds, $\sum_D1\ll s^{o(1)}$ and $\sum_DD\ll s^{1+o(1)}$. For $\nu\le7$, use $|G_{2^\nu}(\mathbf b)|\le2^{5\nu}\le128^32^{2\nu}$ and impose no condition at two; this gives the same estimates with an absolute constant. This proves the estimate for all powerful $s$. Finally $B^3s^{1/2}\le B^4$ for $B\ge s^{1/2}$.
\end{proof}

We apply the next proposition with $\gamma=1/2$ for Theorem~\ref{thm:density-one-main}, and with $\gamma$ just above $1/5$ in Section~\ref{sec:allq-three-eighths}.

\begin{prop}\label{prop:small-nsf-moment}
Fix $0\le\gamma\le1$ and $\sigma>0$, and let $0<\theta_1,\theta_2<1/2$ satisfy
$$
4\theta_1+6\theta_2+2\sigma<3,\qquad
4\theta_1+4\theta_2+2\sigma+\gamma<3,\qquad
2\theta_1+\theta_2+\sigma<1.
$$
Then, for every sufficiently large admissible modulus $q$ with $P(q)\le q^\gamma$, the moment formulae of Proposition~\ref{prop:qw-transfer} hold in each nonempty parity class.
\end{prop}

\begin{proof}
Let $s=P(q)$. Every auxiliary modulus $q_2$ contains $s$, as its complement is squarefree and coprime to it (Lemma~\ref{lem:primitive-factorisation}). Proposition~\ref{prop:nsf-noloss} and $s\le q^\gamma$ give, for every dyadic length $M_2$,
$$
A_{q_2}(2M_2)\ll_{\varepsilon_1}q^{\varepsilon_1}\bigl(q_2^{5/2}M_2^4+q^{\gamma/2}q_2^{5/2}M_2^3+q_2^3M_2^2\bigr).
$$

Let $\Delta=1-(2\theta_1+\theta_2+\sigma)>0$. Since $\theta_1+\theta_2<1-\Delta$, after the dyadic losses $M_1M_2\le q^{1-\Delta/2}$ for large $q$. Lemma~\ref{lem:fixed-factorisation-large-ratio} saves a power in the summed large-ratio range. In its complement,
$$
X=\frac{M_1N_2}{N_1}<q^{-1+\sigma}M_1^2M_2\le q^{-\Delta+o(1)},\qquad\text{so}\qquad X^{1/4}\le q^{-1/4+\sigma/4}M_1^{1/2}M_2^{1/4}.
$$

Insert the three terms of the moment bound separately into the bound for $k\ne0$ in Proposition~\ref{prop:qw-preinsertion}, and use $q_2\le q$ and $X\le1$. For example, the cubic term gives
$$
q^{-3/2}M_2^{-1/2}X^{-1/2}(Xq_2(1+X))^{3/4}\bigl(q^{\gamma/2}q_2^{5/2}M_2^3\bigr)^{1/4}\ll q^{-1/8+\gamma/8}M_2^{1/4}X^{1/4}.
$$
Apart from arbitrarily small powers of $q$, the three contributions are therefore at most
$$
q^{-3/8+\sigma/4}M_1^{1/2}M_2^{3/4},\qquad
q^{-3/8+\gamma/8+\sigma/4}M_1^{1/2}M_2^{1/2},\qquad
q^{-1/4+\sigma/4}M_1^{1/2}M_2^{1/4}.
$$

All three increase with both lengths, so it is enough to take $M_1=q^{\theta_1}$ and $M_2=q^{\theta_2}$. Their exponents are then
$$
-\frac{3-(4\theta_1+6\theta_2+2\sigma)}{8},\qquad
-\frac{3-(4\theta_1+4\theta_2+2\sigma+\gamma)}{8},\qquad
-\frac{1-(2\theta_1+\theta_2+\sigma)}{4},
$$
which are negative by hypothesis.

By Proposition~\ref{prop:qw-preinsertion}, the $k=0$ term and the modulus $q_2=1$ contribute
$$
\ll q^{\varepsilon_1}\frac{(M_1M_2)^{1/2}}q\le q^{-1/2-\Delta/4+\varepsilon_1}.
$$
Choose $\varepsilon_1$ and the dyadic losses below a fixed fraction of the least of these savings and of the savings in the preliminary remainder estimates. The $q^{o(1)}$ dyadic boxes, divisors and factorisations, and the factor $q/\varphi^*(q)$, are then absorbed, and Proposition~\ref{prop:fixed-parity-qw-setup} gives both moment formulae in each nonempty parity class.
\end{proof}

\subsection{Density among admissible moduli}

\begin{lem}\label{lem:large-nsf-part-density}
For $0<\gamma<1$ and all $Q\ge2$,
$$
\#\{q\le Q:P(q)>q^\gamma\}\ll_\gamma Q^{1-\gamma/2}.
$$
\end{lem}

\begin{proof}
The estimate $\#\{n\le X:n\hbox{ powerful}\}\ll X^{1/2}$ gives $\sum_{n>Y,\ n\textrm{ powerful}}n^{-1}\ll Y^{-1/2}$ by dyadic summation. For $X<q\le2X$, the condition forces a powerful divisor exceeding $X^\gamma$, so the number of such $q$ is
$$
\ll X\sum_{\substack{n>X^\gamma\\ n\textrm{ powerful}}}\frac1n\ll X^{1-\gamma/2}.
$$
 Sum over dyadic intervals.
\end{proof}

\begin{proof}[Proof of Theorem~\ref{thm:density-one-main}]
Take $0<\rho<1/16$, $\theta_1=3/8-\rho$, $\theta_2=1/4-\rho$ and $\sigma=\rho$. With $\gamma=1/2$,
$$
4\theta_1+6\theta_2+2\sigma=3-8\rho,\qquad
4\theta_1+4\theta_2+2\sigma+\gamma=3-6\rho,\qquad
2\theta_1+\theta_2+\sigma=1-2\rho,
$$
so Proposition~\ref{prop:small-nsf-moment} gives both moment formulae, in each nonempty parity class, for every sufficiently large admissible $q$ with $P(q)\le q^{1/2}$. Corollary \ref{cor:qw-nv} gives
$$
\kappa(q)\ge\frac{5/8-2\rho}{13/8-2\rho}+o(1);
$$
choosing $\rho$ small proves the pointwise assertion, also parity-wise. Finally Lemma \ref{lem:large-nsf-part-density} with $\gamma=1/2$ gives $O(X^{3/4})$ exceptional moduli, and the finitely many small conductors contribute $O_\ep(1)$.
\end{proof}

\section{The uniform $3/8$ theorem for all admissible moduli}
\label{sec:allq-three-eighths}

\subsection{A reciprocal bilinear estimate}
Proposition~\ref{prop:reciprocal-factorable} and its mean-square form \eqref{eq:reciprocal-mean-square} are reciprocal-variable analogues of \cite[Theorems~2.1 and~3.1]{MQW}; the latter refines \cite[Theorem~5]{BM2015}. We follow the differencing argument of \cite[Section~9]{BM2015} and \cite[Section~3.2]{MQW}. Beyond \cite[Section~3.2]{MQW}, the inverted variable requires one observation. Inversion changes the local labels to $k_v\equiv c_1\bar v$, whose differences $c_1\bar v_1\bar v_2(v_2-v_1)$ have the same common divisors with $r_1$ as $v_2-v_1$. The diagonal in $q$-Weyl differencing gives $V^{-1/2}s^{1/2}$ in \eqref{eq:reciprocal-factorable}.

For $Q\ge2$, we use the normalisation
$$
\operatorname{Kl}_2(n;Q)=Q^{-1/2}S(1,n;Q).
$$

\begingroup
\tolerance=1000
\emergencystretch=3em
\begin{prop}\label{prop:reciprocal-factorable}
Let $Q\ge2$, let $U,V\ge1$, let $s\mid Q$, and let $(c,Q)=1$. Let $\mathcal I$ be an integer interval of length $O(U)$ on which $|u|\asymp U$. Let $(a_u)$ be supported in $\mathcal I$, and let $(b_v)$ be supported in an integer interval of length at most $V<Q$. Assume that $(v,Q)=1$ whenever $b_v\ne0$. Then, for every $\ep>0$,
\begin{equation}\label{eq:reciprocal-factorable}
\begin{aligned}
\left|\sum_u\sum_v a_ub_v\operatorname{Kl}_2(cu\bar v;Q)\right|
&\ll_\ep (QUV)^\ep\|a\|_2\|b\|_2(UV)^{1/2}\\
&\quad\times\left(
 V^{-1/2}s^{1/2}
+Q^{-1/4}s^{1/4}
+U^{-1/2}Q^{1/4}s^{-1/4}\right).
\end{aligned}
\end{equation}
As in \cite[Theorem~2.1]{MQW}, the row variable $u$ need not be a unit modulo $Q$, and the estimate is uniform when $Q$ contains an arbitrary power of $2$.
\end{prop}

\begin{proof}
After separating signs of $u$, it suffices to prove
\begin{equation}\label{eq:reciprocal-mean-square}
\sum_{u\in\mathcal I}
 \left|\sum_v b_v\operatorname{Kl}_2(cu\bar v;Q)\right|^2
\ll_\ep (QUV)^\ep\|b\|_2^2
 \left(Us+UVQ^{-1/2}s^{1/2}
       +VQ^{1/2}s^{-1/2}\right).
\end{equation}
First restrict to $(u,Q)=1$. Choose a nonnegative smooth function $W(x)$, of fixed compact support and bounded derivatives, majorising the row interval after scaling by $U$. Write $\widehat W(\xi)=\int_\R W(x)\e(-x\xi)\,dx$ and
$$
 \mathcal T_W^\times=
 \sum_{(u,Q)=1}W(u/U)
 \left|\sum_vb_v\operatorname{Kl}_2(cu\bar v;Q)\right|^2.
$$
Put in $r_2$ the full prime powers of $Q$ that divide $s$, and put the others in $r_1$. Then
\begin{equation}\label{eq:reciprocal-r1-r2-H}
 Q=r_1r_2,\quad (r_1,r_2)=1,\quad s=Hr_2,\quad
 H\mid \frac{r_1^\sharp}{\operatorname{rad}(r_1^\sharp)},
 \qquad
 r_1^\sharp=\prod_{p^\nu\Vert r_1,\ \nu\ge2}p^\nu.
\end{equation}
Here $\operatorname{rad}(n)=\prod_{p\mid n}p$, and $\operatorname{Kl}_2(\cdot;1)=1$. Twisted multiplicativity gives
\begin{equation}\label{eq:reciprocal-Kl-factorisation}
 \operatorname{Kl}_2(cu\bar v;Q)
 =\operatorname{Kl}_2(c_1u\bar v;r_1)
  \operatorname{Kl}_2(c_2u\bar v;r_2),
\end{equation}
where $c_1\equiv c\bar r_2^{2}\bmod {r_1}$ and $c_2\equiv c\bar r_1^{2}\bmod {r_2}$. Write $r_1=\prod_jq_j$, $q_j=p_j^{\nu_j}$, $R_j=r_1/q_j$ and $R_jt_j\equiv1\bmod {q_j}$. The local branches $S^{\epsilon}_{p^\nu}(a,m)$ and $S^{\epsilon}_{2^\nu}(a,m)$ are those of \eqref{eq:reciprocal-odd-branch} and \eqref{eq:reciprocal-two-branch} in Section~\ref{sec:preliminaries}.

Let $\mathcal P$ contain the odd primes with exponent at least two in $r_1$, and also $2$ at exponent $6$ or at least $8$. For $\boldsymbol\epsilon\in\{\pm1\}^{\mathcal P}$, put
$$
 \mathcal B_{q_j,\boldsymbol\epsilon}(a,m)=
 \begin{cases}
  S_{q_j}^{\epsilon_{p_j}}(a,m),&p_j\in\mathcal P,\\
  S(a,m;q_j),&p_j\notin\mathcal P,
 \end{cases}
 \qquad
 \mathcal B_{\boldsymbol\epsilon}(k,m;r_1)
 =\prod_j\mathcal B_{q_j,\boldsymbol\epsilon}(kt_j,mt_j).
$$
Thus
$$
S(k,m;r_1)=\sum_{\boldsymbol\epsilon}\mathcal B_{\boldsymbol\epsilon}(k,m;r_1),
$$
with at most $2^{\omega(r_1)}\ll_\ep r_1^\ep$ branches. For unit $k_1,k_2$ let $\Delta=k_1-k_2$ and
\begin{equation}\label{eq:reciprocal-complete-correlation-definition}
 \mathcal C_{\boldsymbol\epsilon}(h;k_1,k_2;r_1)
 =\sum_{m\bmod r_1}^{\times}
  \mathcal B_{\boldsymbol\epsilon}(k_1,m;r_1)
  \overline{\mathcal B_{\boldsymbol\epsilon}(k_2,m;r_1)}
  \e(-hm/r_1).
\end{equation}
With $\operatorname{ord}_p0=+\infty$, we prove
\begin{equation}\label{eq:reciprocal-global-correlation-bound}
 \left|\sum_{\boldsymbol\epsilon}
  \mathcal C_{\boldsymbol\epsilon}(h;k_1,k_2;r_1)\right|
 \ll_\ep r_1^{3/2+\ep}(\Delta,h,r_1)^{1/2}1_{G\mid h},
 \qquad G\ge G_0/64,
\end{equation}
where $G_0=(\Delta,r_1^\sharp/\operatorname{rad}(r_1^\sharp))$ and $G\mid r_1$. Bounds of this type are proved in \cite[Proposition~23]{BM2015} for odd moduli and in \cite[Lemma~3.4]{MQW} for powers of $2$. We prove \eqref{eq:reciprocal-global-correlation-bound} in the form needed here, with its frequency condition, to keep the paper self-contained.

For $p\Vert r_1$, let
$$
C_p(h;a_1,a_2)=\sum_{m\bmod p}^{\times}S(a_1,m;p)\overline{S(a_2,m;p)}\e(-hm/p)
$$
and $R_p(a)=\sum_{x\bmod p}^{\times}\e(ax/p)$. Opening the sums gives
$$
\begin{aligned}
 C_p(h;a_1,a_2)&=pT_p(h;a_1,a_2)
  -R_p(a_1)\overline{R_p(a_2)},\\
 T_p(h;a_1,a_2)&=
 \sum_{\substack{x\bmod p\\x\ne0,\ 1-hx\ne0}}
 \e\left(\frac{a_1x-a_2x/(1-hx)}p\right)\quad(h\ne0),\\
 C_p(0;a_1,a_2)&=pR_p(a_1-a_2)
  -R_p(a_1)\overline{R_p(a_2)}.
\end{aligned}
$$
For $h\ne0$, $a_1=a_2=0$ gives $C_p=-1$; otherwise the phase is nonconstant and Weil's bound applies, with linear cases summed directly. In detail, if $a_2\not\equiv0$ the phase is a rational function with a pole at $x=\bar h$, so $|T_p|\le2p^{1/2}+1$ by Weil's bound, after omitting $x=0$; if $a_2\equiv0\not\equiv a_1$ the sum is linear and $|T_p|\le2$; and $|R_p(a_1)R_p(a_2)|\le p$ unless $a_1\equiv a_2\equiv0$. For $h=0$ the last formula gives $|C_p(0;a_1,a_2)|\le p^2$, and $|C_p(0;a_1,a_2)|\le2p$ when $a_1\not\equiv a_2$. Together with the zero-frequency formula, this gives
\begin{equation}\label{eq:reciprocal-squarefree-bound}
 |C_p(h;a_1,a_2)|\ll
 p^{3/2}(a_1-a_2,h,p)^{1/2}.
\end{equation}
At $p=2$ use direct summation. This also covers $3^1$; the following odd-prime-power argument starts at exponent two.

For a branched $q=p^\nu$ define
$$
 C_{q,\epsilon}(h;a_1,a_2)=
 \sum_{m\bmod q}^{\times}S_q^\epsilon(a_1,m)
 \overline{S_q^\epsilon(a_2,m)}\e(-hm/q);
$$
for unbranched powers use the full factors and write $C_q$. Let $p$ be odd and $\nu\ge2$. If $a_1/a_2$ is nonsquare modulo $p$, the supports are disjoint. In the aligned case $a_1\equiv a_2\bmod p$, choose a representative $u_0$ of their common supporting square class and put $\alpha_i=(a_iu_0)_{1/2}$. Compatibility gives $\alpha_1\equiv\alpha_2\bmod p$. On the support write $m=u_0x_m^2$, with $x_m=(a_1m)_{1/2}\bar\alpha_1$; the chosen roots are $\alpha_ix_m$. The product of the two conjugate Gauss factors is $\Phi_{p,\nu}(\alpha_1,\alpha_2)$, where $\Phi_{p,\nu}(A,B)=1$ for even $\nu$ and $(AB/p)$ for odd $\nu$. It is independent of $m,\epsilon$, including when $p\equiv3\bmod4$, since the factors $i$ cancel. Pairing $x_m$ and $-x_m$ therefore gives
\begin{equation}\label{eq:reciprocal-paired-odd}
 C_{q,+}+C_{q,-}
 =q\Phi_{p,\nu}(\alpha_1,\alpha_2)\sum_{x\bmod q}^{\times}
  \e\left(\frac{2(\alpha_1-\alpha_2)x-hu_0x^2}{q}\right).
\end{equation}

For
$$
T_\nu(A,B)=\sum_{x\bmod p^\nu}^{\times}\e\Bigl(\frac{Ax-Bx^2}{p^\nu}\Bigr),
$$
translation by $p^{\nu-\operatorname{ord}_pB-1}$ in each non-zero class modulo $p$ gives zero if $\operatorname{ord}_pB$ is less than both $\operatorname{ord}_pA$ and $\nu-1$. Also
\begin{equation}\label{eq:reciprocal-restricted-gauss-bound}
 |T_\nu(A,B)|\le
 2p^{\nu/2}(A,B,p^\nu)^{1/2}.
\end{equation}
Indeed, remove $p^\delta$, where $\delta=\min(\operatorname{ord}_pA,\operatorname{ord}_pB,\nu)$. The case $\delta=\nu$ is trivial. Otherwise the residual sum has a unit coefficient. If its quadratic coefficient is a unit, subtract the multiples of $p$ from the full quadratic Gauss sum; each is bounded by the square root of the residual modulus, since the full sum has absolute value $p^{(\nu-\delta)/2}$. The sum over multiples of $p$ is $1$ when $\nu-\delta=1$; at larger exponents it vanishes if the residual linear coefficient is a unit, and otherwise is $p$ times a quadratic Gauss sum modulo $p^{\nu-\delta-2}$, of absolute value $p^{(\nu-\delta)/2}$. If only its linear coefficient is a unit, translation makes the sum zero at exponent at least two; at exponent one it is $-1$. Multiplication by $p^\delta$ gives the bound, including vanishing coefficients.

Since $\alpha_1+\alpha_2$ is a unit, the congruence
$$
(\alpha_1-\alpha_2)(\alpha_1+\alpha_2)\equiv u_0(a_1-a_2)\bmod {p^\nu}
$$
identifies the truncated valuations of $\alpha_1-\alpha_2$ and $a_1-a_2$. Thus \eqref{eq:reciprocal-paired-odd} and \eqref{eq:reciprocal-restricted-gauss-bound} give
\begin{equation}\label{eq:reciprocal-odd-local-bound}
 |C_{q,+}+C_{q,-}|
 \ll p^{3\nu/2}(a_1-a_2,h,p^\nu)^{1/2}
 1_{p^{\min\{\operatorname{ord}_p(a_1-a_2),\nu-1\}}\mid h}.
\end{equation}
For the unaligned case $a_1\not\equiv a_2\bmod p$, choose $\beta^2\equiv a_2\bar a_1\bmod q$. There is a sign $\sigma(m)$, depending only on $m\bmod p$, with $(a_2m)_{1/2}=\sigma(m)\beta(a_1m)_{1/2}$. The map ${(m,\epsilon)\mapsto x=\epsilon(a_1m)_{1/2}}$ is a bijection onto the units, with inverse $m\equiv \bar a_1x^2$. Hence
$$
 C_{q,+}+C_{q,-}
 =q\sum_{\sigma=\pm1}\Phi_{p,\nu}(1,\sigma\beta)
  \sum_{x\in X_\sigma}
  \e\left(\frac{2(1-\sigma\beta)x-h\bar a_1x^2}{q}\right),
$$
where $X_+,X_-$ partition the units into unions of classes modulo $p$. Both $1\pm\beta$ are units. For each sign the derivative $2(1-\sigma\beta)-2h\bar a_1x$ vanishes on at most one class modulo $p$; all other classes sum to zero by translation by $p^{\nu-1}$. On a surviving class, $h$ is a unit; writing $x=x_0+py$ gives $p$ times a quadratic Gauss sum modulo $p^{\nu-2}$, bounded by $p^{\nu/2}$, also for $\nu=2$. Hence $|C_{q,+}+C_{q,-}|\le2q^{3/2}$. This is \eqref{eq:reciprocal-odd-local-bound}, since $\operatorname{ord}_p(a_1-a_2)=0$, and includes $p=3$.

For $q=2^\nu$, $\nu=6$ or $\nu\ge8$, the supports are disjoint unless $a_1\equiv a_2\bmod8$. Define $r$ by $2^{r}=(a_1-a_2,2^\nu)$, so $3\le r\le\nu$. If $r=\nu$, the chosen branches agree and have squared modulus $2^{\nu+1}$ on $m\equiv a_1\bmod8$. Direct summation gives
\begin{equation}\label{eq:reciprocal-two-diagonal-endpoint}
 C_{2^\nu,\epsilon}(h;a_1,a_1)
 =2^{2\nu-2}\e(-ha_1/2^\nu)
   1_{2^{\nu-3}\mid h}.
\end{equation}

For $r<\nu$ we use the paired root expansion of \cite[Lemma~3.2]{MQW}. If $u,u'\equiv1\bmod8$, $u\equiv u'\bmod {2^{\lambda+1}}$, $t\equiv t'\bmod {2^{\lambda}}$ and $j\ge3$, then
\begin{align}\label{eq:reciprocal-paired-root}
 &(u+2^jt)_{1/2}-(u'+2^jt')_{1/2}\notag\\
 &\quad\equiv u_{1/2}-u'_{1/2}
 +2^{j-1}(\bar u_{1/2}t-\bar u'_{1/2}t')
 \pmod{2^{2j+\lambda-3}}.
\end{align}
For completeness, put
$$
F_n(u,t)=\sum_{\ell=0}^n\binom{1/2}{\ell}u_{1/2}^{1-2\ell}(2^jt)^\ell.
$$
  The binomial identity and $\operatorname{ord}_2(\ell!)\le\ell-1$ give
$$
 \operatorname{ord}_2(F_n(u,t)^2-u-2^jt)\ge(n+1)(j-2)+2,\quad
 (u+2^jt)_{1/2}\equiv F_n(u,t)\pmod{2^{(n+1)(j-2)+1}}.
$$
The hypotheses give $u_{1/2}\equiv u'_{1/2}\bmod {2^{\lambda}}$. The difference of degree-$\ell$ terms, $\ell\ge2$, has valuation at least $(j-2)\ell+\lambda+1\ge2j+\lambda-3$; the linear difference is the term shown in \eqref{eq:reciprocal-paired-root}. Take $n$ sufficiently large.

Set
\begin{equation}\label{eq:reciprocal-two-Tj}
 T=\left\lfloor\frac{\nu+r}{2}\right\rfloor-3,
 \qquad j=\nu-T;\qquad j\ge4,\quad 2j+r-4>\nu.
\end{equation}
Write $m=m_1+2^{j}m_2$ with $m_1\bmod2^{j}$, $m_2\bmod2^{T}$, and put $\rho_i=(m_1a_i)_{1/2}$. Apply \eqref{eq:reciprocal-paired-root} with $(u,u',t,t',\lambda)=(m_1a_1,m_1a_2,a_1m_2,a_2m_2,r-1)$. On the common support $m_1\equiv a_1\equiv a_2\bmod8$, so all hypotheses hold. The error is beyond the phase modulus and the Gauss factors are independent of $m_2$ since $j\ge4$. Orthogonality makes the $m_2$-sum zero unless
\begin{equation}\label{eq:reciprocal-two-orthogonality}
 L(m_1):=a_1\bar\rho_1-a_2\bar\rho_2
 \equiv\epsilon h\pmod{2^{T}},
\end{equation}
in which case its size is $2^{T}$. Since $a_i\bar\rho_i\equiv \rho_i\bar m_1\bmod {2^\nu}$ and $\operatorname{ord}_2(\rho_1+\rho_2)=1$, the identity
$$
(\rho_1-\rho_2)(\rho_1+\rho_2)\equiv m_1(a_1-a_2)\bmod {2^\nu}
$$
gives $\operatorname{ord}_2L(m_1)=r-1$.

\looseness=-1 If $3\le r\le\nu-6$, then $r-1<T$ and $h=2^{r-1}h'$ with $h'$ odd. Choose odd $w\equiv a_1\bmod8$ and set $A_i=(a_iw)_{1/2}$, $x=(m_1w)_{1/2}$. Comparison of squares and residues modulo $4$ gives
$$
a_i\overline{(m_1a_i)_{1/2}}=\sigma A_i\bar x
$$
with the same sign $\sigma$ for both $i$, and $\operatorname{ord}_2(A_1-A_2)=r-1$. For $\beta=2^{-(r-1)}(A_1-A_2)$ and $M=T-r+1$, \eqref{eq:reciprocal-two-orthogonality} becomes $h'x\equiv\sigma\epsilon\beta\bmod {2^{M}}$. Here $1\le M\le j$, since $2T\le\nu+r-6$. At most two root classes, hence at most two classes of $m_1\bmod2^{M}$, contribute. Thus there are $O(2^{j-M})=O(2^{\nu-2T+r-1})$ choices and
$$
 |C_{2^\nu,\epsilon}|
 \ll2^{\nu+1}2^{\nu-2T+r-1}2^{T}
 \ll2^{3\nu/2}(2^{r-1})^{1/2}.
$$
On this support $(a_1-a_2,h,2^\nu)=2^{r-1}$.

If $\nu-5\le r<\nu$, then $r-3\le T\le r-1$ and $T\ge\nu-6$. The same orthogonality forces $2^{T}\mid h$. Trivial counting gives
$$
|C_{2^\nu,\epsilon}|\ll2^{2\nu}\le8\cdot2^{3\nu/2}(a_1-a_2,h,2^\nu)^{1/2},
$$
and $2^{r}\mid8h$. Together with \eqref{eq:reciprocal-two-diagonal-endpoint}, this proves
\begin{equation}\label{eq:reciprocal-two-local-bound}
 |C_{2^\nu,\epsilon}(h;a_1,a_2)|
 \ll 2^{3\nu/2}(a_1-a_2,h,2^\nu)^{1/2}
 1_{(a_1-a_2,2^\nu)\mid8h}.
\end{equation}

For $\nu\in\{2,3,4,5,7\}$ retain the full factor. Then $|C_{2^\nu}|\le2^{3\nu}\ll2^{3\nu/2}(a_1-a_2,h,2^\nu)^{1/2}$ with an absolute constant, and we impose no frequency condition. The exponent-one factor contributes nothing to $G_0$. Define
$$
 G_{\rm odd}=\prod_{\substack{p^\nu\Vert r_1\\p>2,\ \nu\ge2}}
 p^{\min(\operatorname{ord}_p\Delta,\nu-1)}.
$$
At a branched two-power multiply this by $2^{\max\{\min(\operatorname{ord}_2\Delta,\nu)-3,0\}}$; at other two-powers use no extra factor. The resulting $G$ satisfies $G\ge G_0/64$, since the exceptional two-power loses at most $2^6$, and a branched one at most $2^3$. CRT, with $m\mapsto t_jm$ locally, gives
\begin{equation}\label{eq:reciprocal-correlation-product}
 \sum_{\boldsymbol\epsilon}
 \mathcal C_{\boldsymbol\epsilon}
 =\prod_{\substack{p^\nu\Vert r_1\\p\in\mathcal P}}
  (C_{p^\nu,+}+C_{p^\nu,-})
  \prod_{\substack{p^\nu\Vert r_1\\p\notin\mathcal P}}
   C_{p^\nu},
\end{equation}
where each local correlation has arguments $(h;k_1t_p,k_2t_p)$. The units $t_p$ preserve every valuation and gcd. The bounds \eqref{eq:reciprocal-squarefree-bound}, \eqref{eq:reciprocal-odd-local-bound} and \eqref{eq:reciprocal-two-local-bound} therefore prove \eqref{eq:reciprocal-global-correlation-bound}. The simultaneous branch sum is kept until this product, since taking absolute values branch by branch loses \eqref{eq:reciprocal-paired-odd}.

For unit $k_1,k_2$, put
$$
 \Sigma^{\boldsymbol\epsilon}(k_1,k_2)
 =\sum_{(u,Q)=1}W(u/U)
  \mathcal B_{\boldsymbol\epsilon}(k_1,u;r_1)
  \overline{\mathcal B_{\boldsymbol\epsilon}(k_2,u;r_1)}.
$$
Invert only $(u,r_2)=1$, retaining $(u,r_1)=1$. Writing $u=dn$ for $d\mid\operatorname{rad}(r_2)$ gives
$$
\Sigma^{\boldsymbol\epsilon}(k_1,k_2)=\sum_{d\mid\operatorname{rad}(r_2)}\mu(d)\Sigma_d^{\boldsymbol\epsilon}(k_1,k_2),
$$
where
$$
\Sigma_d^{\boldsymbol\epsilon}(k_1,k_2)
=\sum_{(n,r_1)=1}W(dn/U)
\mathcal B_{\boldsymbol\epsilon}(k_1,dn;r_1)
\overline{\mathcal B_{\boldsymbol\epsilon}(k_2,dn;r_1)}.
$$
Poisson summation gives
\begin{equation}\label{eq:reciprocal-poisson-completion}
 \Sigma_d^{\boldsymbol\epsilon}
 =\frac{U}{dr_1}\sum_{h\in\Z}
  \widehat W\left(-\frac{hU}{dr_1}\right)
 \mathcal C_{\boldsymbol\epsilon}(\bar d h;k_1,k_2;r_1).
\end{equation}
Since $d$ is a unit modulo $r_1$, the gcd and the condition $G\mid h$ are unchanged by $\bar d$. Put $D=(\Delta,r_1)$. The zero frequency contributes $\ll(U/d)r_1^{1/2+\ep}D^{1/2}$; for the other frequencies rapid decay gives $\sum_{h\ne0,\ G\mid h}|\widehat W(-hU/(dr_1))|\ll dr_1/(UG)$. Using $G\ge G_0/64$ and the divisor bounds for $\sum_d1$ and $\sum_d1/d$, we obtain
\begin{equation}\label{eq:reciprocal-completed-correlation}
 \left|\sum_{\boldsymbol\epsilon}
  \Sigma^{\boldsymbol\epsilon}(k_1,k_2)\right|
 \ll_\ep (Qr_1)^\ep
 \left(Ur_1^{1/2}D^{1/2}
       +\frac{r_1^{3/2}D^{1/2}}{G_0}\right).
\end{equation}

We now bound the unit-row mean square. Weil--Estermann gives $|\operatorname{Kl}_2(n;M)|\ll_\ep M^\ep$ for $(n,M)=1$, including all two-powers. If $V<s$, Cauchy's inequality gives $\ll Q^\ep UV\|b\|_2^2\le Q^\ep Us\|b\|_2^2$. Otherwise apply Cauchy's inequality in the $s$ residue classes of $v$. Within each class the $r_2$-factor in \eqref{eq:reciprocal-Kl-factorisation} is constant in $v$; its square costs $r_2^\ep$. For $k_v\equiv c_1\bar v\bmod {r_1}$, the remaining factor is $r_1^{-1/2}S(k_v,u;r_1)$. Cauchy's inequality over branch vectors, before opening the square, costs $r_1^\ep$.

The diagonal contributes $\ll(QUV)^\ep Us\|b\|_2^2$, and the off-diagonal is at most
$$
 \frac{s}{r_1}(QUV)^\ep
 \sum_{\substack{v_1\ne v_2\\v_1\equiv v_2\bmod s}}
 |b_{v_1}b_{v_2}|
 \left|\sum_{\boldsymbol\epsilon}
 \Sigma^{\boldsymbol\epsilon}(k_{v_1},k_{v_2})\right|.
$$
For every $E\mid r_1$, inversion gives $(k_{v_1}-k_{v_2},E)=(v_2-v_1,E)$. Write $v_2-v_1=\ell s=\ell Hr_2$, $0<|\ell|\ll V/s$, and $R=r_1/H$. Then $D=H(\ell,R)$ and $G_0\ge H$ by \eqref{eq:reciprocal-r1-r2-H}. For each $\ell$, $\sum_v|b_vb_{v+\ell s}|\le\|b\|_2^2$, while
$$
 \sum_{0<|\ell|\ll V/s}(\ell,R)^{1/2}
 \le\sum_{d\mid R}d^{1/2}\#\{0<|\ell|\ll V/s:d\mid\ell\}
 \ll_\ep (V/s)r_1^\ep.
$$
Insert these bounds in \eqref{eq:reciprocal-completed-correlation}. The off-diagonal becomes
$$
 \ll_\ep (QUV)^\ep\|b\|_2^2 V
 \left(U\sqrt{H/r_1}+\sqrt{r_1/H}\right)
 =(QUV)^\ep\|b\|_2^2
 \left(UVQ^{-1/2}s^{1/2}
       +VQ^{1/2}s^{-1/2}\right),
$$
since $H/r_1=s/Q$. Together with the diagonal, this proves \eqref{eq:reciprocal-mean-square} for unit rows.

To remove the unit-row restriction, note that if $p^N\Vert Q$, $N\ge2$, and $p\mid u$, the local factor vanishes. Indeed, in $S(1,n;p^N)$ with $p\mid n$, translation $x\mapsto x+tp^{N-1}$ changes the phase by $tp^{N-1}(1-n\bar x^2)$, a nontrivial character in $t\bmod p$. This also treats $p=2$ by pairing the two translates. Thus only squarefree unitary $d=(u,Q)$ survive. With $Q_d=Q/d$ and $u=du'$, CRT gives a unit $c_d\bmod Q_d$ and
\begin{equation}\label{eq:reciprocal-nonunit-scaling}
 \operatorname{Kl}_2(cdu'\bar v;Q)
 =\frac{\mu(d)}{d^{1/2}}
  \operatorname{Kl}_2(c_du'\bar v;Q_d),
\end{equation}
since each removed prime factor is $p^{-1/2}S(1,0;p)=-p^{-1/2}$.

If $Q_d=1$, Cauchy's inequality gives $(UV/Q^2+V/Q)\|b\|_2^2$, absorbed by the first and third terms of \eqref{eq:reciprocal-mean-square}. If $U/d<1$, there are $O(1)$ scaled rows; Weil and Cauchy give $\ll Q^\ep V\|b\|_2^2$, absorbed by the third term since $s\le Q$. Otherwise $U/d\ge1$.

Aggregate columns by $b_d(r)=\sum_{v\equiv r\bmod Q_d}b_v$ and split a cyclic support into at most two ordinary intervals of length $V_d<Q_d$. Then
$$
 V_d\ll\min(V,Q_d),\quad
 \|b_d\|_2^2\le(1+V/Q_d)\|b\|_2^2,\quad
 (1+V/Q_d)V_d\ll V.
$$
Apply the unit estimate at modulus $Q_d$ with lengths $U/d,V_d$ and $s_d=s/(s,d)$. Since $V<Q$ and $Q=dQ_d$, inclusion of the factor $d^{-1}$ from the square of \eqref{eq:reciprocal-nonunit-scaling} bounds the ratios of its three terms to the desired terms by
$$
 \frac{1+V/Q_d}{d^2(s,d)},\qquad
 \frac{2}{d^{3/2}(s,d)^{1/2}},\qquad
 \frac{2(s,d)^{1/2}}{d^{3/2}},
$$
respectively. These are bounded, and divisor summation is absorbed by $(QUV)^\ep$. This proves \eqref{eq:reciprocal-mean-square} for all rows.

Finally apply Cauchy's inequality against $a_u$ and take square roots of the three terms in \eqref{eq:reciprocal-mean-square} to obtain \eqref{eq:reciprocal-factorable}.
\end{proof}
\endgroup

\subsection{A complete ratio operator bound}
All matrix norms below are operator norms on $\ell^2$ with counting measure. For a finite matrix $T$, we write
$$
\|T\|_{S^4}^4=\operatorname{tr}\bigl((TT^*)^2\bigr).
$$
In particular, $\|T\|\le\|T\|_{S^4}$.
\begin{lem}\label{lem:complete-ratio-operator}
Let $\mathcal U,\mathcal V$ be subsets of integer intervals of lengths $U,V$, let $(c,Q)=1$, and suppose that every $v\in\mathcal V$ is a unit modulo $Q$. The integer-indexed matrix
$$
A_Q(u,v)=\operatorname{Kl}_2(cu\bar v;Q)
$$
satisfies
\begin{equation}\label{eq:complete-ratio-operator}
\|A_Q\|
\le Q^{1/2}
\left(1+\frac UQ\right)^{1/2}
\left(1+\frac VQ\right)^{1/2}.
\end{equation}
\end{lem}

\begin{proof}
On complete residue spaces, changing the Kloosterman variable gives
$$
A_Q(u,v)=Q^{-1/2}\sum_{x\bmod Q}^{\times}
\e(ux/Q)\e(c\bar x\bar v/Q).
$$
The first factor has orthonormal columns; the second is a submatrix of the unnormalised additive Fourier matrix. Thus the complete operator norm is at most $Q^{1/2}$, including nonunit $u$. Aggregation into residue classes has norms at most $(1+U/Q)^{1/2}$ and ${(1+V/Q)^{1/2}}$ on the two sides, giving \eqref{eq:complete-ratio-operator}.
\end{proof}

\subsection{The local prime-square estimate}

\begin{lem}\label{lem:p2-short-root-difference}
Let $p$ be an odd prime, let $1\le H<p$, and let $I$ be a cyclic interval modulo $p^2$ of length at most $CpH$, where $C$ is fixed. Let $c$ be a unit modulo $p^2$, and let
$$
X=\{x\bmod p^2:(x,p)=1,\ cx^2\in I\bmod {p^2}\}.
$$
For $d\bmod p^2$, write
$$
r_X(d)=\#\{(x_1,x_2)\in X^2:x_1-x_2=d\bmod {p^2}\}.
$$
If $p\nmid d$, then
\begin{equation}\label{eq:p2-short-root-difference}
r_X(d)\ll_{C,\ep}p^\ep
\left(H^2+p^{1/2}H^{3/2}\right).
\end{equation}
\end{lem}

\begin{proof}
Replacing $H$ by $\lceil H\rceil\le2H$ changes only the implied constant. If $\lceil H\rceil\ge p$, the trivial bound suffices; otherwise we may assume that $H$ is an integer.

If $H\ge p/(2C)$, then $r_X(d)\le |X|\ll_CpH$ suffices, since then $pH\le2CH^2$. Otherwise $CpH<p^2/2$, after changing fixed constants.

We first show, for unit $t\bmod p^2$, an interval $K\subset[-CpH,CpH]$ and a cyclic interval $L$ of length at most $CpH$,
\begin{equation}\label{eq:modular-parabola-count}
\#\{k\in K:(k^2+t^2)/(2t)\in L\bmod {p^2}\}
\ll_{C,\ep}p^\ep(H^2+p^{1/2}H^{3/2}).
\end{equation}

Split $L$ into $O_C(1)$ arcs and take $L=L_0+[0,pH)$. Let $\lambda(z)=\lfloor[z-L_0]_{p^2}/p\rfloor$, where $[\cdot]_{p^2}$ is the least nonnegative residue. Then $z\in L$ exactly when $\lambda(z)<H$. For $0\le j<p$ let $N_j$ be the number of $k\in K$ with $\lambda((k^2+t^2)/(2t))=j$. Then $\sum_jN_j=|K|$, and $\sum_jN_j^2=E$ is the number of pairs $k_1,k_2\in K$ with equal values of $\lambda((k_i^2+t^2)/(2t))$. Write $N_j=|K|/p+(N_j-|K|/p)$ and apply Cauchy's inequality to the second part over $j<H$. Since $\sum_j(N_j-|K|/p)^2\le\sum_jN_j^2$, this gives
\begin{equation}\label{eq:last-digit-fourier}
\#\{k\in K:(k^2+t^2)/(2t)\in L\}
\le |K|H/p+(H E)^{1/2}.
\end{equation}

Each pair counted by $E$ satisfies $(k_1-k_2)(k_1+k_2)\equiv2t z\bmod {p^2}$ for some $|z|<p$. For the representative $T\in[1,p^2]$ of $t$, put $A=k_1-k_2$, $B=k_1+k_2$ and lift to
\begin{equation}\label{eq:parabola-divisor-lift}
AB=2T z+\ell p^2.
\end{equation}
For fixed $z$, $|A|,|B|\ll_CpH$ restrict $\ell$ to an interval of length $O_C(H^2+1)$. A non-zero right side has $O_\ep(p^\ep)$ factor pairs; if it vanishes, then $z=0$, because $t$ is a unit and $|z|<p$, and $A=0$ or $B=0$, giving $O_C(pH)$ pairs. Summing over the $O(p)$ values of $z$ gives $E\ll_{C,\ep}p^{1+\ep}(H^2+H)$. Since $|K|\ll_CpH$, \eqref{eq:last-digit-fourier} now gives \eqref{eq:modular-parabola-count}.

For a pair counted by $r_X(d)$, choose $n,m\in I$ representing $cx_1^2,cx_2^2$ and put $k=n-m$, $\ell=n+m$, $t=cd^2$. Since $d$ is a unit,
$$
x_1+x_2\equiv k/(cd),\qquad \ell\equiv(k^2+t^2)/(2t)\pmod{p^2}.
$$
The sets $I-I,I+I$ split into $O_C(1)$ intervals of the required form. Given $d,k$, the sum and difference determine $(x_1,x_2)$ uniquely. Thus \eqref{eq:modular-parabola-count} gives the result.
\end{proof}

\begin{prop}\label{prop:p2-unequal-local-operator}
Let $p$ be an odd prime. Let $I_1,I_2$ be cyclic intervals modulo $p^2$ of lengths at most $CpH_1$ and $CpH_2$, where $1\le H_1,H_2<p$ and $C$ is fixed. Let $c_1,c_2$ be units and let
$$
X=\{x:(x,p)=1,\ c_1x^2\in I_1\},
\qquad
Y=\{y^{-1}:(y,p)=1,\ c_2y^2\in I_2\}.
$$
Let $R$ be the exponential matrix
$$
R(x,y)=\e(2xy/p^2),\qquad x\in X,\quad y\in Y.
$$
Then
\begin{equation}\label{eq:p2-unequal-schatten}
\|R\|_{S^4}^4
\ll_{C,\ep}p^\ep
\left(
p^3H_1H_2^2+p^3H_2H_1^2
+p^{7/2}H_2H_1^{3/2}\right).
\end{equation}
Consequently, the compression to $I_1\times I_2$ of the normalised matrix
$$
p^{-1}S(1,cu\bar v;p^2)
$$
has operator norm at most a constant times the fourth root of the right hand side of \eqref{eq:p2-unequal-schatten}, uniformly in the unit $c$ and in multiplicative dilates of the two intervals.
\end{prop}

\begin{proof}
Write $\widehat{1_Y}(h)=\sum_{y\in Y}\e(hy/p^2)$. Since $(RR^*)(x_1,x_2)=\widehat{1_Y}(2(x_1-x_2))$ and
$$
\operatorname{tr}\bigl((RR^*)^2\bigr)=\sum_{x_1,x_2}|(RR^*)(x_1,x_2)|^2,
$$
grouping the pairs by $d=x_1-x_2$ gives
\begin{equation}\label{eq:p2-schatten-identity}
\|R\|_{S^4}^4
=\sum_{d\bmod p^2}r_X(d)|\widehat{1_Y}(2d)|^2.
\end{equation}
\looseness=-1 Each target in $I_i$ has at most two unit square roots, so $|X|\ll_CpH_1$ and $|Y|\ll_CpH_2$. For $x$ in a fixed class modulo $p$, the values $c_1x^2$ run through distinct elements of one class modulo $p$, and $I_1$ contains $O_C(H_1)$ of them; the same holds for $Y$. Thus the residue fibres of $X$ and $Y$ modulo $p$ have sizes $O_C(H_1)$ and $O_C(H_2)$. For $p\mid d$, use $r_X(d)\le |X|$ and
$$
\sum_{a\bmod p}|\widehat{1_Y}(2pa)|^2
=p\sum_{s\bmod p}\#\{y\in Y:y\equiv s\bmod p\}^2
\ll_Cp^2H_2^2,
$$
\looseness=-1 where the last step bounds one factor of each square by the fibre size and sums the other to $|Y|$. This gives $O_C(p^3H_1H_2^2)$. For unit $d$, Lemma \ref{lem:p2-short-root-difference} and Parseval modulo $p^2$ give
$$
\ll_{C,\ep}p^\ep(H_1^2+p^{1/2}H_1^{3/2})p^2|Y|
\ll_{C,\ep}p^\ep(p^3H_2H_1^2+p^{7/2}H_2H_1^{3/2}).
$$
This proves \eqref{eq:p2-unequal-schatten}.

For the Kloosterman matrix, rows with $p\mid u$ vanish, and $p^{-1}S(1,n;p^2)=\sum_{r^2\equiv n}\e(2r/p^2)$ for units $n$, by Lemma~\ref{lem:p2-kloosterman-eval}. Split the rows and columns into square classes, $u=c_1x^2$ and $v=c_2y^2$. Then $cu\bar v$ is a square only if $cc_1\bar c_2=\gamma^2$ for a unit $\gamma$, and in that case the roots of $r^2\equiv cu\bar v$ are $\pm\gamma x\bar y$. Replacing $x$ by $\gamma x$ and $c_1$ by $c_1\bar\gamma^2$, which is allowed by the uniformity in $c_1$, we may take $\gamma=1$. Fixing one root $y=y(v)$ for each column, the entry is
$$
p^{-1}S(1,cu\bar v;p^2)=\sum_{x:\ c_1x^2=u}\e\Bigl(\frac{2x\bar y}{p^2}\Bigr).
$$

The block is $ARB$: $A$ sends $f$ to $u\mapsto\sum_{c_1x^2=u}f(x)$ and has norm at most $\sqrt2$, while $B$ is the isometric inclusion $v\mapsto\bar y(v)$. Summing the four class pairs and using the fourth Schatten norm gives $\ll\|R\|_{S^4}$.
\end{proof}

\subsection{CRT operator bound for $Q=p^2e$}

\begin{prop}\label{prop:p2-crt-operator}
Let $Q=p^2e$, where $p$ is odd and $(p,e)=1$. Let $\mathcal U,\mathcal V\subset\Z$ be supported in intervals of lengths at most $U$ and $V$, respectively, and suppose that every element of $\mathcal V$ is a unit modulo $Q$. Assume $U\le Q/8$ and $V\le p^2/8$, and let
$$
H_1=1+\frac{U}{ep},
\qquad
H_2=1+\frac{V}{p}.
$$
Replacing $H_1,H_2$ by their ceilings is understood. For any unit $c\bmod Q$, the matrix
$$
T_Q(u,v)=\operatorname{Kl}_2(cu\bar v;Q),
\qquad
u\in\mathcal U,\quad v\in\mathcal V,
$$
satisfies
\begin{equation}\label{eq:p2-crt-operator}
\|T_Q\|
\ll_\ep Q^\ep e^{1/2}
\left(
p^3H_1H_2^2+p^3H_2H_1^2
+p^{7/2}H_2H_1^{3/2}\right)^{1/4}.
\end{equation}
The estimate is valid for arbitrary complex row and column vectors. In particular, no factorability assumption is imposed on their coefficients.
\end{prop}

\begin{proof}
Let $c_1=c\bar e^{2}\bmod {p^2}$ and $c_2=c\overline{p^2}^{2}\bmod e$, and define the complete local matrices
$$
T_{p^2}(u,v)
=p^{-1}S(1,c_1u\bar v;p^2),
\qquad
T_e(u,v)
=e^{-1/2}S(1,c_2u\bar v;e),
$$
with arbitrary rows and unit columns. CRT identifies $T_Q$ with a compression of $T_{p^2}\otimes T_e$; the unit twists are covered by the uniformity in Proposition \ref{prop:p2-unequal-local-operator}. The complete-space argument of Lemma \ref{lem:complete-ratio-operator} gives
\begin{equation}\label{eq:cofactor-norm}
\|T_e\|\le e^{1/2},
\end{equation}
also for nonunit rows and arbitrary two-powers in $e$.

Let $\iota_{\mathcal V}$ be CRT inclusion on the column support and $\pi_{\mathcal U}$ row restriction. They have norm at most one, since $U,V<Q$ makes both maps injective on the supports. Set $S=(I\otimes T_e)\iota_{\mathcal V}$, so $\|S\|\le e^{1/2}$. This operator leaves the $p^2$-coordinate unchanged. Hence, if $J=\mathcal V\bmod p^2$, the range of $S$ is supported on $J\times\Z/e\Z$, and $J$ lies in a cyclic interval of length $O(V)$. With $\Pi_J$ the corresponding projection,
$$
T_Q=\pi_{\mathcal U}(T_{p^2}\otimes I)\Pi_JS.
$$
Decompose the rows as $\mathcal U=\bigsqcup_{i\bmod e}\mathcal U_i$. The first three factors form the direct sum
$$
\pi_{\mathcal U}(T_{p^2}\otimes I)\Pi_J
=\bigoplus_{i\bmod e}\pi_{\mathcal U_i}T_{p^2}\Pi_J.
$$
Its norm is the largest block norm. Modulo $p^2$, each $\mathcal U_i$ is contained in an arithmetic progression of step $e$ and length $O(1+U/e)$. A unit dilation makes this a cyclic interval, of length $\ll pH_1$; the column length is $\ll pH_2$. Proposition \ref{prop:p2-unequal-local-operator} bounds each block by the fourth root in \eqref{eq:p2-crt-operator}. Multiply by $\|S\|\le e^{1/2}$.
\end{proof}

\subsection{Operator bounds for the Poisson-summed cross term}

\begin{prop}\label{prop:allq-qw-operator-insertion}
Let $q=q_1Q$, with $(q_1,Q)=1$ and $\mu^2(q_1)=1$, be one auxiliary factorisation in Proposition \ref{prop:qw-preinsertion}. Let $d_0\mid q_1$ denote the M\"obius-inversion divisor in
$$
\mathbf 1_{(n_1,q_1)=1}
=\sum_{d_0\mid(n_1,q_1)}\mu(d_0),
$$
so that $n_1=d_0n$. Let
$$
z=\theta_1+\theta_2.
$$
Outside the range of Lemma \ref{lem:fixed-factorisation-large-ratio}, each dyadic block with $k\ne0$ can be grouped into
$$
u=kn_2,\qquad v=m_1m_2,
$$
\looseness=-1 with dyadic size parameters $U\asymp KN_2$, $V\asymp M_1M_2$ on $|k|\asymp K$, chosen so that, after separating signs, the supports lie in intervals of lengths at most $U$ and $V$, with $|u|\asymp U$, and
\begin{equation}\label{eq:allq-grouped-ranges}
U\ll d_0Qq^{-1+z+\sigma},
\qquad
V\ll q^{z}.
\end{equation}
The grouped coefficients have the divisor-bound estimates
$$
\|a\|_2\ll q^\ep U^{1/2},
\qquad
\|b\|_2\ll q^\ep V^{1/2},
$$
and every $v$ is a unit modulo $Q$.

For every $Q>1$, Lemma \ref{lem:complete-ratio-operator} gives
\begin{equation}\label{eq:allq-complete-operator-block}
|\mathcal B_{\pm,d_0}(q_1,Q)|
\ll q^\ep
\left(\frac Qq\right)^{3/2}d_0^{-1/2}
\left(1+\frac UQ\right)^{1/2}
\left(1+\frac VQ\right)^{1/2}.
\end{equation}

If $s\mid Q$ and $V<Q$, Proposition \ref{prop:reciprocal-factorable} gives
\begin{equation}\label{eq:allq-factorable-block}
\begin{aligned}
|\mathcal B_{\pm,d_0}(q_1,Q)|
&\ll q^\ep\frac{Q}{q\sqrt {d_0}}
\sqrt{\frac{UV}{q}}\\
&\quad\times
\left(
V^{-1/2}s^{1/2}
+Q^{-1/4}s^{1/4}
+U^{-1/2}Q^{1/4}s^{-1/4}\right).
\end{aligned}
\end{equation}

Suppose instead that $Q=p^2e$, where $p$ is odd and $(p,e)=1$, and that $U\le Q/8$ and $V\le p^2/8$. Then Proposition \ref{prop:p2-crt-operator} gives
\begin{equation}\label{eq:allq-p2-block}
|\mathcal B_{\pm,d_0}(q_1,Q)|
\ll q^\ep
\frac{Qe^{1/2}}{q^{3/2}d_0^{1/2}}
\left(
p^3H_1H_2^2+p^3H_2H_1^2
+p^{7/2}H_2H_1^{3/2}\right)^{1/4},
\end{equation}
where
$$
H_1=1+\frac{U}{ep},
\qquad
H_2=1+\frac{V}{p}.
$$
\end{prop}

\begin{proof}
In \eqref{eq:poisson-block}, write the M\"obius divisor as $d_0$. Scaling and symmetry give
$$
S(k,\pm n_2\overline{d_0m_1m_2q_1};Q)
=Q^{1/2}\operatorname{Kl}_2(cu\bar v;Q),
\qquad
u=kn_2,\quad v=m_1m_2,\quad c=\pm\overline{d_0q_1}.
$$
Rapid decay gives $|k|\ll d_0Qq^\ep/N_1$. Separate signs and dyadic $|k|\asymp K$, and group by the integer products. The divisor bound gives the stated coefficient norms; the complementary-range inequality gives \eqref{eq:allq-grouped-ranges}. With $U\asymp KN_2$, $V\asymp M_1M_2$, so that $M_1M_2N_1N_2\asymp UVN_1/K$, and with $\varphi(Q)\le Q$ and $|F^{(j)}_{d_0}(k)|\ll q^\ep$ in \eqref{eq:poisson-block}, the coefficient of the normalised bilinear sum is bounded by
$$
\frac{Q}{q^{3/2}(d_0UV)^{1/2}}
\left(\frac{KN_1}{d_0Q}\right)^{1/2}
\ll\frac{q^\ep Q}{q^{3/2}(d_0UV)^{1/2}}.
$$
The two coefficient norms cancel $(UV)^{-1/2}$. Applying Lemma \ref{lem:complete-ratio-operator}, Proposition \ref{prop:reciprocal-factorable}, or Proposition \ref{prop:p2-crt-operator} then gives respectively \eqref{eq:allq-complete-operator-block}, \eqref{eq:allq-factorable-block}, or \eqref{eq:allq-p2-block}.
\end{proof}

\begin{lem}\label{lem:allq-p2-exponents}
Assume the hypotheses of Proposition \ref{prop:allq-qw-operator-insertion} and that $Q=p^2e$. Write
$$
Q=q^{t},\qquad p=q^{a},\qquad d_0=q^{b}.
$$
Thus $0<2a\le t\le1$ and $0\le b\le1-t$. Assume $a<z$ and use the upper bounds \eqref{eq:allq-grouped-ranges}.

If $U/(ep)\ge1$, the three terms in \eqref{eq:allq-p2-block} have exponents
\begin{align}
E_1&=\frac32t-\frac12a+\frac34z
-\frac74-\frac14b+\frac14\sigma,
\label{eq:allq-p2-e1}\\
E_2&=\frac32t+\frac34z-2+\frac12\sigma,
\label{eq:allq-p2-e2}\\
E_3&=\frac32t+\frac58z-\frac{15}{8}
-\frac18b+\frac38\sigma.
\label{eq:allq-p2-e3}
\end{align}
If $U/(ep)<1$, the corresponding exponents are
\begin{align}
E'_1&=\frac32t-\frac32-\frac12b
-\frac34a+\frac12z,
\label{eq:allq-p2-e4}\displaybreak[0]\\
E'_2&=\frac32t-\frac32-\frac12b
-\frac12a+\frac14z,
\label{eq:allq-p2-e5}\displaybreak[0]\\
E'_3&=\frac32t-\frac32-\frac12b
-\frac38a+\frac14z.
\label{eq:allq-p2-e6}
\end{align}
All six exponents increase with $t=\log_qQ$ and are nonincreasing in $b=\log_qd_0$; hence their maximum occurs at $Q=q$ and $d_0=1$. The first list is therefore power-saving if
\begin{equation}\label{eq:allq-p2-large-row-conditions}
3z+\sigma<1+2a,\qquad
3z+2\sigma<2,\qquad
5z+3\sigma<3,
\end{equation}
and the second is power-saving if
\begin{equation}\label{eq:allq-p2-small-row-conditions}
2z<3a,\qquad
z<2a.
\end{equation}
\end{lem}

\begin{proof}
Enlarge the column length to $q^{z}$, so $H_2\ll q^{z-a}$ since $a<z$. The scalar in \eqref{eq:allq-p2-block} has exponent $3t/2-a-3/2-b/2$. Writing $h_i=\log_qH_i$, the three terms under the fourth root have exponents $3a+h_1+2h_2$, $3a+2h_1+h_2$ and $\frac72a+\frac32h_1+h_2$. If $U/(ep)\ge1$, use $H_1\ll q^{b+z+a-1+\sigma}$; if $U/(ep)<1$, use $H_1\ll1$. Substitution, division by $4$ and addition of the exponent of the scalar gives the two displayed lists, including their common transition. Every expression increases with $t$ and is nonincreasing in $b$, so setting $t=1$, $b=0$ gives the stated sufficient inequalities.
\end{proof}

\subsection{Cases determined by $P(q)$ and its largest prime divisor}

Recall that $P(q)$ denotes the powerful part of $q$.

\begin{prop}\label{prop:allq-three-eighths-moment}
Fix $\rho>0$ sufficiently small and let
\begin{equation}\label{eq:allq-length-choice}
\theta_1=\frac25-\rho,
\qquad
\theta_2=\frac15-\rho,
\qquad
z=\theta_1+\theta_2=\frac35-2\rho.
\end{equation}
\looseness=-1 Choose $0<\sigma<\rho/4$. Then, for every sufficiently large $q\not\equiv2\bmod4$, the moment formulae of Proposition \ref{prop:qw-transfer} hold, for each nonempty parity class, with the lengths in \eqref{eq:allq-length-choice}.
\end{prop}

\begin{proof}
We have
\begin{equation}\label{eq:allq-preliminary-identity}
2\theta_1+\theta_2=1-3\rho.
\end{equation}
The summed large-ratio range saves a power by Lemma \ref{lem:fixed-factorisation-large-ratio}. In its complement,
$$
X\ll q^{-1+\sigma+2\theta_1+\theta_2+o(1)}\ll q^{-2\rho},
$$
so Proposition \ref{prop:allq-qw-operator-insertion} applies to every non-zero frequency. We separate three cases.

\emph{Small powerful part.} If $P(q)\le q^{1/5+\rho/2}$, apply Proposition \ref{prop:small-nsf-moment} with $\gamma=1/5+\rho/2$. Indeed, $4\theta_1+6\theta_2+2\sigma<14/5-19\rho/2<3$, $4\theta_1+4\theta_2+2\sigma+\gamma<13/5-7\rho<3$ and $2\theta_1+\theta_2+\sigma<1-11\rho/4$.

\emph{Intermediate prime factors.} Suppose $P(q)>q^{1/5+\rho/2}$ and every prime factor of $P(q)$ is at most $q^{2/5+\rho}$. Then some $s\mid P(q)$ satisfies
\begin{equation}\label{eq:allq-medium-divisor}
q^{1/5+\rho/2}<s\le q^{2/5+\rho}.
\end{equation}
Take a prime in this interval if one exists. Otherwise all prime factors are at most $q^{1/5+\rho/2}$; multiply them with multiplicity until this endpoint is exceeded. The last factor and the preceding product are each at most the endpoint, so their product is at most its square.

Every auxiliary $Q$ contains $P(q)$, since its complement $q_1$ is squarefree and coprime to $Q$ (Lemma~\ref{lem:primitive-factorisation}), hence $s\mid Q$. Put $\delta=\rho/10$. Since $d_0Q\le q$, because $d_0\mid q_1$, we have $U,V\ll q^{z+\sigma}$. If $Q\le q^{1-\delta}$, the complete operator bound gives
\begin{equation}\label{eq:allq-proper-auxiliary}
|\mathcal B_{\pm,d_0}(q_1,Q)|
\ll q^\ep
\left(\frac Qq\right)^{3/2}
\left(1+\frac{q^{z+\sigma}}Q\right).
\end{equation}
For $Q\ge q^{z+\sigma}$ this saves $q^{-3\delta/2+o(1)}$; otherwise its second term is at most $q^{-3(1-z-\sigma)/2}$ and its first is also power-saving.

\looseness=-1 For $Q>q^{1-\delta}$, we have $U,V<Q$ and apply the factorable estimate with $s=q^{w}$. Write $t=\log_qQ$ and $b=\log_qd_0$. Substituting \eqref{eq:allq-grouped-ranges} in \eqref{eq:allq-factorable-block}, the three terms have exponents
$$
\tfrac32t+\tfrac12z-2+\tfrac12w+\tfrac12\sigma,\qquad
\tfrac54t+z-2+\tfrac14w+\tfrac12\sigma,\qquad
\tfrac54t+\tfrac12z-\tfrac32-\tfrac14w-\tfrac12b,
$$
up to arbitrarily small powers of $q$. These are increasing in $t$ and nonincreasing in $b$; at $t=1$ and $b=0$ they are negative precisely when
\begin{equation}\label{eq:allq-factorable-conditions}
z+w+\sigma<1,\qquad
4z+w+2\sigma<3,\qquad
2z<1+w.
\end{equation}
For $1/5+\rho/2<w\le2/5+\rho$, these follow from
$$
z+w+\sigma<1-3\rho/4,\quad
4z+w+2\sigma<14/5-13\rho/2<3,\quad
1+w-2z>9\rho/2.
$$

\emph{A large prime divisor.} Suppose $p>q^{2/5+\rho}$ divides $P(q)$. Its exponent is two, since $p^3>q$. Write $p=q^{a}$, with $a>2/5+\rho$; every auxiliary modulus is $Q=p^2e$, $(p,e)=1$. The grouped ranges give $V\ll q^{z}$ and $U\ll q^{z+\sigma}$. Since $2a-z>1/5+4\rho$ and $2a-z-\sigma>1/5+15\rho/4$, we have $U\le Q/8$, $V\le p^2/8$ for sufficiently large $q$. Apply Proposition \ref{prop:p2-crt-operator} and Lemma \ref{lem:allq-p2-exponents}. In the large-row regime,
\begin{gather*}
3z+\sigma<9/5-23\rho/4<1+2a,
\qquad
3z+2\sigma<9/5-11\rho/2<2,\\
5z+3\sigma<3-37\rho/4;
\end{gather*}
in the small-row regime, $2z=6/5-4\rho<3a$ and $z<2a$. Thus every auxiliary block saves a fixed power of $q$.

Finally the $k=0$ and $Q=1$ terms are $\ll q^\ep(M_1M_2)^{1/2}/q\le q^{-7/10-\rho+o(1)}$. Choose the exponent and dyadic losses below half the least positive saving. The $q^{o(1)}$ divisors, boxes, signs and factorisations are absorbed. Proposition \ref{prop:fixed-parity-qw-setup} then yields the two moments in both parity classes.
\end{proof}

\subsection{Completion of the proof of Theorem~\ref{thm:allq-three-eighths}}
\label{sec:proof-allq-main}

\begin{proof}[Proof of Theorem~\ref{thm:allq-three-eighths}]
Proposition \ref{prop:allq-three-eighths-moment} and Corollary \ref{cor:qw-nv} give
$$
\kappa(q),\ \kappa_j(q)\ge\frac{3/5-2\rho}{8/5-2\rho}+o(1)
$$
for every nonempty parity class. Choose $\rho$ small in terms of $\ep$ and then $q$ sufficiently large.
\end{proof}

Ignoring the arbitrarily small $\sigma$-loss, the factorable conditions permit $z<f(w)$, where
$$
f(w)=\min\left\{1-w,\frac{3-w}{4},\frac{1+w}{2}\right\}.
$$
They are simultaneously tight at $w=1/3$, giving $z<2/3$. The divisor construction above only guarantees $w\in[\gamma,2\gamma]$, and $\max_{\gamma}\min\{f(\gamma),f(2\gamma)\}=3/5$, attained at $\gamma=1/5$. Thus changing the split point alone cannot improve this divisor argument; a divisor near $q^{1/3}$ gives the following refinement.

\subsection{Moduli whose powerful part has a divisor near $q^{1/3}$}\label{sec:third-divisor}

\begin{cor}\label{cor:third-divisor}
For every $\ep>0$ there is $\eta>0$ such that every sufficiently large admissible modulus $q$ whose powerful part has a divisor $s$ with
$$
q^{1/3-\eta}\le s\le q^{1/3+\eta}
$$
satisfies $\kappa(q)\ge\frac25-\ep$.
\end{cor}

\begin{proof}
Choose $0<\omega<1/100$ with $(2/3-2\omega)/(5/3-2\omega)>2/5-\ep$, and put
$$
\theta_1=1/4-\omega,\quad
\theta_2=5/12-\omega,\quad
z=2/3-2\omega,\quad
\eta=\sigma=\omega/2,\quad
\delta=\omega/10.
$$
Then $2\theta_1+\theta_2=11/12-3\omega<1$. Write $s=q^{w}$ with $|w-1/3|\le\eta$. Each auxiliary $Q$ contains $s$. The preliminary argument of Proposition \ref{prop:allq-three-eighths-moment} applies with these lengths. The large-ratio range is summed over factorisations, and in the complement ${X\ll q^{-1/12-5\omega/2+o(1)}}$.

For $Q\le q^{1-\delta}$, \eqref{eq:allq-proper-auxiliary} gives savings $q^{-3\omega/20}$ or $q^{-1/2-9\omega/4}$ according as $Q\ge q^{z+\sigma}$ or not. The $k=0$ and $Q=1$ terms are $\ll q^{-2/3-\omega+o(1)}$. For $Q>q^{1-\delta}$, $V<Q$, and the factorable exponents at their worst endpoint $Q=q,d_0=1$ are
$$
\begin{gathered}
-\omega+(w-1/3)/2+\sigma/2\le-\omega/2,\qquad
-2\omega+(w-1/3)/4+\sigma/2\le-13\omega/8,\\
-\omega-(w-1/3)/4\le-7\omega/8.
\end{gathered}
$$
These savings absorb the small exponent losses, preliminary remainders and $q^{o(1)}$ divisors and dyadic boxes as in Proposition~\ref{prop:allq-three-eighths-moment}. Proposition \ref{prop:fixed-parity-qw-setup} and Corollary \ref{cor:qw-nv} give $\kappa(q)\ge z/(1+z)+o(1)>2/5-\ep$ for sufficiently large $q$.
\end{proof}

\begin{proof}[Proof of Theorem~\ref{thm:two-fifths}]
Here $P(q)=p^2$ and $q=p^{2+b}$ with $|b-4|\le\eta$, so $p^2=q^{2/(2+b)}$ with $|2/(2+b)-1/3|\ll\eta$. Apply Corollary \ref{cor:third-divisor} with $s=p^2$.
\end{proof}

These moduli also satisfy $P(q)\le q^{1/2}$ when $\eta<2$, but the present constant improves $5/13$; Theorem~\ref{thm:hybrid-interpolation} does not apply, since it requires $r\le p^a$ with $a<1$. More generally, a divisor $s=q^{w}$ gives $\kappa(q)\ge f(w)/(1+f(w))-\ep$, exceeding $5/13$ for $1/4<w<3/8$. For $q=rp^2$ with $r=p^{b}$ and $s=p^2$, this is $10/3<b<6$.

\section{Mixed conductors}\label{sec:hybrid}

\subsection{A complete moment uniform in the squarefree cofactor}

\begin{prop}\label{prop:hybrid-uniform-exceptional}
Fix $\lambda,\delta_0>0$ and $0<\xi<\min(\lambda,1)$. There is a set of primes $\mathcal E$ satisfying
$$
\#\{p\le X:p\in\mathcal E\}
\ll_{\lambda,\delta_0,\xi}\frac{X^{1-\xi}}{\log X},
$$
such that, for every $\ep>0$, every odd prime $p\notin\mathcal E$, every squarefree $d$ with $p\nmid d$, and every $B\le p^{1/2-\delta_0}$, one has
$$
A_{dp^2}(B)\ll_{\lambda,\delta_0,\xi,\ep}
(dp^2)^\ep\left(d^{1/2}p^\lambda B^4(dp^2)^{5/2}
+p^\lambda B^2(dp^2)^3\right).
$$
\end{prop}

\begin{proof}
Choose $\mathcal E$ from Theorem~\ref{thm:p2-average} with $\ep=\lambda$, $\delta=\delta_0$ and the given $\xi$, before choosing $d,B$ or $\ep$. Lemma \ref{lem:kloo-mult} gives $G_{dp^2}(\mathbf b)=G_d(\mathbf b)G_{p^2}(\mathbf b)$, and Lemma \ref{lem:sqfree-pointwise} with $D_d(\mathbf b)\le d$ gives $|G_d(\mathbf b)|\ll_\ep d^{3+\ep}$. Drop $(b_i,d)=1$, retaining $(b_i,p)=1$, to obtain, simultaneously for all permitted lengths,
$$
A_{dp^2}(B)\ll_\ep d^{3+\ep}A_{p^2}(B)\ll_\ep d^{3+\ep}p^\lambda(B^4p^5+B^2p^6).
$$
Now $d^3p^5=d^{1/2}(dp^2)^{5/2}$ and $d^3p^6=(dp^2)^3$; absorb $d^\ep$ into $(dp^2)^\ep$.
\end{proof}

\begin{lem}\label{lem:hybrid-exponents}
Fix $a_0\ge0$, $\theta_1,\theta_2>0$, and $\lambda,\sigma\ge0$. Let $q=rp^2$ with $r\le p^{a_0}$, let $d\mid r$ and set
$$
s=\frac{\log r}{\log p},\qquad
t=\frac{\log d}{\log p}.
$$
Suppose that, for every sufficiently small $\varepsilon_2>0$,
$$
A_{dp^2}(B)\ll (dp^2)^{\varepsilon_2}
\left(d^{1/2}p^\lambda B^4(dp^2)^{5/2}
+p^\lambda B^2(dp^2)^3\right).
$$
If $M_1\le q^{\theta_1+o(1)}$, $M_2\le q^{\theta_2+o(1)}$, and
$$
X\ll q^\sigma \frac{M_1^2M_2}{q},
$$
then
$$
X\ll q^{-1+2\theta_1+\theta_2+\sigma+o(1)}.
$$
If $2\theta_1+\theta_2+\sigma<1$, define
$$
\begin{aligned}
E_4(s,t)&=
\frac{-(3+7s-6t)
+(2+s)(2\theta_1+3\theta_2+\sigma)+\lambda}{4},\\
E_2(s,t)&=
\frac{-(2+7s-6t)
+(2+s)(2\theta_1+\theta_2+\sigma)+\lambda}{4}.
\end{aligned}
$$
For every $\varepsilon_3>0$, after choosing $\varepsilon_2$ and the dyadic $o(1)$-loss sufficiently small in terms of $\varepsilon_3$, the two contributions from the terms with $k\ne0$ are respectively $O(p^{E_4(s,t)+\varepsilon_3})$ and $O(p^{E_2(s,t)+\varepsilon_3})$. For fixed $s$, both exponents increase with $t$.
\end{lem}

\begin{proof}
The complementary-range bound gives $X\ll q^{-1+2\theta_1+\theta_2+\sigma+o(1)}$, and hence $X<1$ for large $q$ when $2\theta_1+\theta_2+\sigma<1$. Insert the assumed moment in the bound for $k\ne0$ in Proposition \ref{prop:qw-preinsertion}, with $q_2=dp^2$ and $B=2M_2$. Since $X<1$, the quartic and quadratic terms, before arbitrarily small powers, are
$$
q^{-3/2}X^{1/4}M_2^{1/2}(dp^2)^{11/8}d^{1/8}p^{\lambda/4},
\qquad
q^{-3/2}X^{1/4}(dp^2)^{3/2}p^{\lambda/4}.
$$
For example, the quartic term gives
$$
q^{-3/2}M_2^{-1/2}X^{-1/2}(Xdp^2)^{3/4}\bigl(d^{1/2}p^{\lambda}M_2^4(dp^2)^{5/2}\bigr)^{1/4},
$$
which is the first expression. Use $X^{1/4}\ll q^{-1/4+\sigma/4}M_1^{1/2}M_2^{1/4}$, $q=p^{2+s}$ and $dp^2=p^{2+t}$ to obtain $E_4(s,t),E_2(s,t)$. Choose $\varepsilon_2$ and the dyadic losses so their combined exponent is at most $\varepsilon_3$. Both expressions increase with $t$, whose coefficient is $6/4$.
\end{proof}

\begin{prop}\label{prop:hybrid-transfer}
Let $a_0\ge0$ and let $\theta_1,\theta_2,\lambda,\sigma>0$ be fixed. Let $q=rp^2$, where $r$ is odd and squarefree, $p\nmid r$, and $r\le p^{a_0}$. Suppose that a set of exceptional primes has been fixed independently of $\varepsilon_2$, and that for every sufficiently small $\varepsilon_2>0$ and every auxiliary modulus $dp^2$ with $d\mid r$, and $B=2M_2$, one has
$$
A_{dp^2}(B)\ll (dp^2)^{\varepsilon_2}
\left(d^{1/2}p^\lambda B^4(dp^2)^{5/2}
+p^\lambda B^2(dp^2)^3\right).
$$
Set $s=\log r/\log p$ and $t=\log d/\log p$. Assume that, for some $\delta>0$ independent of $d\mid r$,
$$
\begin{aligned}
2\theta_1+\theta_2+\sigma&\le1-\delta,\\
(2+s)(2\theta_1+3\theta_2+\sigma)+\lambda
&\le3+7s-6t-\delta,\\
(2+s)(2\theta_1+\theta_2+\sigma)+\lambda
&\le2+7s-6t-\delta.
\end{aligned}
$$
Then the moment formulae of Proposition \ref{prop:qw-transfer} hold for each parity. It is enough to verify the last two inequalities at $t=s$.
\end{prop}

\begin{proof}
The first inequality gives $M_1M_2\le q^{1-\delta+o(1)}$ and $X<q^{-1+\sigma}M_1^2M_2\le q^{-\delta+o(1)}$ in the complementary range. Lemma \ref{lem:fixed-factorisation-large-ratio} treats the summed large-ratio range, and Proposition \ref{prop:qw-preinsertion} treats zero frequencies and modulus one. For non-zero frequencies, Lemma \ref{lem:hybrid-exponents} gives $E_4(s,t),E_2(s,t)\le-\delta/4$. Choose $\varepsilon_3<\delta/8$, and then the smaller losses prescribed there. Each block is $O(p^{-\delta/8})$, uniformly for $0\le t\le s\le a_0$. Since $q=p^{2+s}$ with bounded $s$, this absorbs the $q^{o(1)}$ divisors and dyadic boxes. Proposition \ref{prop:fixed-parity-qw-setup} gives both moments. The case $a_0=0$ has $r=d=1$, $s=t=0$.
\end{proof}

\subsection{Optimisation of the mollifier lengths}

\begin{proof}[Proof of Theorem~\ref{thm:hybrid-interpolation}]
Choose $\omega>0$ small and set
$$
\theta_2=\frac{1}{2(2+a)}-\omega,\qquad
\theta_1=\frac{1}{2}-\frac{1}{4(2+a)}-\omega.
$$
Then $\theta_1+\theta_2=1/2+1/(4(2+a))-2\omega$, $2\theta_1+\theta_2=1-3\omega$ and $2\theta_1+3\theta_2=1+1/(2+a)-5\omega$. The length bound is $B\ll p^{(2+a)\theta_2+o(1)}=p^{1/2-(2+a)\omega+o(1)}$; choose $0<\delta_0<(2+a)\omega/2$. Choose $\lambda,\sigma>0$ with $\lambda+(2+a)\sigma<3\omega$ and $0<\xi_1=\xi<\min(\lambda,1)$; fix $\mathcal E$ from Proposition \ref{prop:hybrid-uniform-exceptional}.

It suffices to check $t=s$, by Lemma \ref{lem:hybrid-exponents}. At $t=s$ the last two inequalities of Proposition~\ref{prop:hybrid-transfer} require
$$
(2+s)(1-3\omega+\sigma)+\lambda<2+s,
\qquad
(2+s)\Bigl(1+\frac1{2+a}-5\omega+\sigma\Bigr)+\lambda<3+s.
$$
\looseness=-1 For $0\le s\le a$, the quadratic inequality follows from $(2+s)(3\omega-\sigma)-\lambda\ge{\Delta_2=6\omega-\lambda-(2+a)\sigma>0}$. For the quartic inequality, the difference between right and left sides is
$$
1-\frac{2+s}{2+a}+(2+s)(5\omega-\sigma)-\lambda.
$$
Since $(2+s)/(2+a)\le1$, it is at least $10\omega-\lambda-(2+a)\sigma>7\omega$. Take the common slack smaller than $\min(3\omega-\sigma,\Delta_2,7\omega)$. Proposition \ref{prop:hybrid-transfer} and Corollary \ref{cor:qw-nv} give, uniformly for $r\le p^a$ outside $\mathcal E$,
$$
\kappa(rp^2)\ge
\frac{\theta_1+\theta_2}{1+\theta_1+\theta_2}+o(1)
=
\frac{2a+5}{6a+13}-O_a(\omega)+o(1).
$$
Take $\omega$ sufficiently small.
\end{proof}

\section*{Acknowledgements}

B.~Durkan and A.~Pearce-Crump gratefully acknowledge support from the Heilbronn Institute for Mathematical Research. B.~Durkan thanks Hung Bui and Xiaosheng Wu for their comments on an earlier version of this paper.

\section*{Declaration on the use of AI}
AI (ChatGPT and Claude) has been used throughout this paper in the following ways. AI was used to help develop ideas for the arguments in Sections 8,9, and 10, changing previous results that were stated as `almost all' to `sufficiently large'. While these built on our ideas, these sections and their corresponding theorems should therefore be read as AI written and proved. It was then used as an adversarial checker of the maths throughout this paper, and to revise the writing throughout the paper, including rewording passages and changing terminology; much of the final wording reflects these revisions. The authors are responsible for the content of the paper.


\begin{thebibliography}{99}

\bibitem{BM}
R. Balasubramanian and V. Kumar Murty,
\emph{Zeros of Dirichlet $L$-functions},
Ann. Sci. \'Ecole Norm. Sup. (4) 25 (1992), no.~5, 567--615.

\bibitem{BM2015}
V. Blomer and D. Mili\'cevi\'c,
\emph{The second moment of twisted modular $L$-functions},
Geom. Funct. Anal. 25 (2015), no.~2, 453--516.

\bibitem{Bui}
H. M. Bui,
\emph{Non-vanishing of Dirichlet $L$-functions at the central point},
Int. J. Number Theory 8 (2012), no.~8, 1855--1881.

\bibitem{BPRZ}
H. M. Bui, K. Pratt, N. Robles, and A. Zaharescu, \emph{Breaking the $1/2$-barrier for the twisted second moment of Dirichlet $L$-functions}, Adv. Math. 370 (2020), article~107175, 40 pp.

\bibitem{BPZ}
H. M. Bui, K. Pratt, and A. Zaharescu, \emph{Exceptional characters and nonvanishing of Dirichlet $L$-functions}, Math. Ann. 380 (2021), no.~1--2, 593--642.

\bibitem{Castillo}
C. Castillo, A. de Faveri, and A. Dunn,
\emph{Non-vanishing for quartic Hecke $L$-functions and ranks of elliptic curves},
arXiv:2604.01316 [math.NT], 2026.

\bibitem{CechMat}
M. \v{C}ech and K. Matom\"aki,
\emph{A note on exceptional characters and non-vanishing of Dirichlet
$L$-functions},
Math. Ann. 389 (2024), no.~1, 987--996.

\bibitem{CechMatOptimal}
M. \v{C}ech and K. Matom\"aki, \emph{On optimality of mollifiers}, arXiv:2501.12526v3 [math.NT], 2025.

\bibitem{Conrey}
J. B. Conrey, \emph{More than two fifths of the zeros of the Riemann zeta function are on the critical line}, J. Reine Angew. Math. 399 (1989), 1--26.

\bibitem{David}
C. David, A. de Faveri, A. Dunn, and J. Stucky,
\emph{Non-vanishing for cubic Hecke $L$-functions},
arXiv:2410.03048 [math.NT], 2024.

\bibitem{Diaconu}
A. Diaconu, B. Ion, V. Pa\c{s}ol, and A. Popa,
\emph{On the second moment and non-vanishing of central values of Hecke $L$-functions of $r$-th order characters},
arXiv:2607.27131 [math.NT], 2026.

\bibitem{DPR}
S. Drappeau, K. Pratt, and M. Radziwi\l\l, \emph{One-level density estimates for Dirichlet $L$-functions with extended support}, Algebra Number Theory 17 (2023), no.~4, 805--830.

\bibitem{FGKM}
\'E. Fouvry, S. Ganguly, E. Kowalski, and P. Michel,
\emph{Gaussian distribution for the divisor function and Hecke eigenvalues
in arithmetic progressions},
Comment. Math. Helv. 89 (2014), no.~4, 979--1014.

\bibitem{FKMS}
\'E. Fouvry, E. Kowalski, P. Michel, and W. Sawin, \emph{Lectures on applied $\ell$-adic cohomology}, in \emph{Analytic Methods in Arithmetic Geometry}, Contemp. Math., vol.~740, American Mathematical Society, Providence, RI, 2019, pp.~113--195.

\bibitem{FI}
J. B. Friedlander and H. Iwaniec, \emph{Incomplete Kloosterman sums and a divisor problem}, Ann. of Math. (2) 121 (1985), no.~2, 319--344.

\bibitem{IK}
H. Iwaniec and E. Kowalski,
\emph{Analytic Number Theory},
American Mathematical Society Colloquium Publications, vol.~53,
American Mathematical Society, Providence, RI, 2004.

\bibitem{IS}
H. Iwaniec and P. Sarnak,
\emph{Dirichlet $L$-functions at the central point},
in \emph{Number Theory in Progress}, vol.~2
(Zakopane--Ko\'scielisko, 1997),
de Gruyter, Berlin, 1999, pp.~941--952.

\bibitem{ISAuto}
H. Iwaniec and P. Sarnak, \emph{The non-vanishing of central values of automorphic $L$-functions and Landau--Siegel zeros}, Israel J. Math. 120 (2000), part~A, 155--177.

\bibitem{KSWX}
B. Kerr, I. E. Shparlinski, X. Wu, and P. Xi, \emph{Bounds on bilinear forms with Kloosterman sums}, J. London Math. Soc. (2) 108 (2023), no.~2, 578--621.

\bibitem{KhanNgo}
R. Khan and H. T. Ngo,
\emph{Nonvanishing of Dirichlet $L$-functions},
Algebra Number Theory 10 (2016), no.~10, 2081--2091.

\bibitem{KMN}
R. Khan, D. Mili\'cevi\'c, and H. T. Ngo,
\emph{Nonvanishing of Dirichlet $L$-functions, II},
Math. Z. 300 (2022), no.~2, 1603--1613.

\bibitem{KMS}
E. Kowalski, P. Michel, and W. Sawin, \emph{Bilinear forms with Kloosterman sums and applications}, Ann. of Math. (2) 186 (2017), no.~2, 413--500.

\bibitem{KMV}
E. Kowalski, P. Michel, and J. VanderKam, \emph{Mollification of the fourth moment of automorphic $L$-functions and arithmetic applications}, Invent. Math. 142 (2000), no.~1, 95--151.

\bibitem{Leung}
S.-K. Leung,
\emph{Non-vanishing of Dirichlet $L$-functions with smooth conductors},
Ramanujan J. 66 (2025), article~69.

\bibitem{MV}
P. Michel and J. VanderKam,
\emph{Non-vanishing of high derivatives of Dirichlet $L$-functions at the
central point},
J. Number Theory 81 (2000), no.~1, 130--148.

\bibitem{MQW}
D. Mili\'cevi\'c, X. Qin, and X. Wu,
\emph{Bilinear forms with Kloosterman sums and moments of twisted
$L$-functions},
arXiv:2511.07550v1 [math.NT], 2025.

\bibitem{Paley}
R. E. A. C. Paley,
\emph{On the $k$-analogues of some theorems in the theory of the Riemann
$\zeta$-function},
Proc. London Math. Soc. (2) 32 (1931), no.~1, 273--311.

\bibitem{QW}
X. Qin and X. Wu,
\emph{Non-vanishing of Dirichlet $L$-functions at the Central Point},
arXiv:2504.11916v2 [math.NT], 2025.

\bibitem{Sound}
K. Soundararajan,
\emph{Nonvanishing of quadratic Dirichlet $L$-functions at $s=1/2$},
Ann. of Math. (2) 152 (2000), no.~2, 447--488.

\bibitem{WuTwisted}
X. Wu, \emph{The twisted mean square and critical zeros of Dirichlet $L$-functions}, Math. Z. 293 (2019), no.~1--2, 825--865.

\bibitem{WuCentral}
X. Wu, \emph{The fourth moment of Dirichlet $L$-functions at the central value}, Math. Ann. 387 (2023), no.~3--4, 1199--1248.

\bibitem{WuLine}
X. Wu, \emph{The fourth moment of Dirichlet $L$-functions along the critical line}, Forum Math. 35 (2023), no.~5, 1347--1371.

\bibitem{Young}
M. P. Young, \emph{The fourth moment of Dirichlet $L$-functions}, Ann. of Math. (2) 173 (2011), no.~1, 1--50.

\end{thebibliography}
\end{document}